\documentclass[english]{amsart}
\usepackage{amssymb,amsmath,amsfonts,amsthm}
\usepackage{stmaryrd}
\usepackage[all,cmtip,poly]{xy}
\usepackage{svn}
\usepackage{enumerate}

\usepackage{mathrsfs}
\usepackage{enumitem} 
\usepackage[latin1]{inputenc}

\usepackage{verbatim}   

\usepackage{amscd}
\usepackage{bbm} 

\usepackage{color}
\definecolor{ghcolor}{RGB}{0, 150, 200} 
\definecolor{winestain}{rgb}{0.5,0,0}
\usepackage[linktocpage=true, hidelinks, colorlinks, linkcolor=blue, citecolor=red, urlcolor=blue]{hyperref} 
\usepackage{url}
\AtEndDocument{{\footnotesize
Homepage: \url{http://bicmr.pku.edu.cn/~gaohui/}
}}

\swapnumbers

\newtheorem{thm}[subsubsection]{Theorem}
\newtheorem{lemma}[subsubsection]{Lemma}

\newtheorem{corollary}[subsubsection]{Corollary}

\newtheorem{prop}[subsubsection]{Proposition}

\theoremstyle{definition}
\newtheorem{defn}[subsubsection]{Definition}
\newtheorem{notation}[subsubsection]{Notation}

\theoremstyle{remark}
\newtheorem{remark}[subsubsection]{Remark}
\newtheorem{rem}[subsubsection]{Remark}

\def\numequation{\addtocounter{subsubsection}{1}\begin{equation}}
\def\nummultline{\addtocounter{subsubsection}{1}\begin{multline}}

\newcommand{\loccit}{\textit{loc. cit.}}

\newcommand{\matkeuu}{\Mat_d(k_E\llb u\rrb)}

\newcommand{\glke}{\GL_d(k_E)}

\newcommand{\M}{\mathfrak{M}}

\newcommand{\hatm}{\hat{\mathfrak{M}}}

\newcommand{\hatN}{\hat{\mathfrak{N}}}
\newcommand{\hatn}{\hat{\mathfrak{N}}}

\newcommand{\barhatm}{\overline{\hat{\mathfrak{M}}}}

\newcommand{\barhatn}{\overline{\hat{\mathfrak{N}}}}

\newcommand{\barm}{\overline{\mathfrak{M}}}
\newcommand{\barn}{\overline{\mathfrak{N}}}

\newcommand{\llb}{\llbracket}  
\newcommand{\rrb}{\rrbracket}

\newcommand{\vshape}{\varphi -\textnormal{shape}}
\newcommand{\vtshape}{(\varphi, \tau) -\textnormal{shape}}

\newcommand{\barchi}{\overline{\chi}}
\newcommand{\barrho}{\overline{\rho}}

\newcommand{\rhobar}{\overline{\rho}}

\newcommand{\rbar}{\overline{r}}

\newcommand{\Zp}{\mathbb{Z}_p}
\newcommand{\Qp}{\mathbb{Q}_p}
\newcommand{\Fp}{\mathbb{F}_p}

\newcommand{\Fpbar}{\overline{\mathbb{F}}_p}
\newcommand{\barK}{\overline{K}}

\newcommand{\Z}{\mathbb{Z}}
\newcommand{\Q}{\mathbb{Q}}

\newcommand{\diag}{\textnormal{diag}}
\newcommand{\col}{\textnormal{col}}

\DeclareMathOperator{\Ext}{Ext}
\DeclareMathOperator{\Fil}{Fil}

\DeclareMathOperator{\Gal}{Gal}
\DeclareMathOperator{\GL}{GL}

\DeclareMathOperator{\Hom}{Hom}

\DeclareMathOperator{\Mat}{Mat}

\DeclareMathOperator{\Mod}{Mod}

\newcommand{\cris}{\mathrm{cris}}

\newcommand{\HT}{\mathrm{HT}}

\newcommand{\D}{\mathcal{D}}
\newcommand{\huaS}{\mathfrak{S}}
\newcommand{\huaM}{\mathfrak{M}}

\newcommand{\huaN}{\mathfrak{N}}
\newcommand{\huaL}{\mathfrak{L}}

\newcommand{\Ghat}{\hat{G}}

\newcommand{\That}{\hat{T}}

\newcommand{\bolde}{\boldsymbol{e}}

\newcommand{\boldt}{\boldsymbol{t}}
\newcommand{\boldf}{\boldsymbol{f}}

\newcommand{\squaremat}[4]{
\left(
 \begin{array}{ccccc}
   #1  & #2\\
   #3 &  #4
 \end{array}
\right)
}

\title{Adapted bases of Kisin modules and Serre weights}
\author{HUI GAO}
\address{Beijing International Center for Mathematical Research, Peking University, No. 5 Yiheyuan Road, Haidian District, Beijing 100871, China}
\email{gaohui@math.pku.edu.cn}

\subjclass[2010]{Primary 11F80, 11F33}
\keywords{torsion Kisin modules, crystalline representations}

\begin{document}

\begin{abstract}
Let $p>2$ be a prime. Let $K$ be a tamely ramified finite extension over $\Qp$ with ramification index $e$, and let $G_K$ be the Galois group. We study Kisin modules attached to crystalline representations of $G_K$ whose labeled Hodge-Tate weights are relatively small (a sort of ``$er\le p$" condition where $r$ is the maximal Hodge-Tate weight). In particular, we show that these Kisin modules admit ``adapted bases". We then apply these results in the special case $e=2$ to study reductions and liftings of certain crystalline representations. As a consequence, we establish some new cases of weight part of Serre's conjectures (when $e=2$).

\end{abstract}

\maketitle
\pagestyle{myheadings}
\markright{Adapted bases of Kisin modules and Serre weights}

\tableofcontents


\section{Introduction} \label{section: intro}

\renewcommand{\O}{\mathcal{O}}
\newcommand{\ve}{\varepsilon} 
\newcommand{\fS}{\mathfrak{S}}

\newcommand{\gr}{{\rm gr}}
\newcommand{\fijd}{{\Fil ^{\{m_{ij }\}}\D}}

\newcommand{\mij}{\{m_{ij}\}}

\newcommand{\mfixedi}{\{m_{i,0},\ldots,m_{i,e-1}\}}


\newcommand{\mf}{\mathfrak}  

\subsection{Motivations and overview}
\subsubsection{Weight part of Serre's conjectures}
Let $F$ be an imaginary CM field and let $G_F$ be the Galois group. Suppose $\rbar: G_F \to \GL_d(\Fpbar)$ is an irreducible representation that is automorphic. The weight part of Serre's conjectures asks for which weights is $\rbar$ automorphic of.
Weight part of Serre's conjectures have proved to be useful in the quest for a $p$-adic local Langlands correspondence, but known results have restrictions in either the dimension $d$ or the ramification of $p$ in $F$.
Conjecturally, the automorphic weights of $\rbar$ should be determined by the information in $\rbar |_{G_{F_v}}$ (or even furthermore, by $\rbar |_{I_v}$), where $v$ runs through all places of $F$ above $p$, $F_v$ the completion at $v$ with Galois group $G_{F_v}$, and $I_v$ the inertia subgroup. When $d=2$ (with $p>2$ and in the unitary group setting), this conjecture is completely proved by \cite{GLS14, GLS15}; in particular, they have shown that for a given Serre weight, it is automorphic if and only if $\rbar |_{G_{F_v}}, \forall v|p$ admits a crystalline lift with Hodge-Tate weights corresponding to the given Serre weight.

However, once the dimension $d$ gets bigger than $2$, we have rather little evidence to see if the above picture (automorphic weight $\Leftrightarrow$ crystalline weight) still holds completely. Note that one direction (automorphic weight $\Rightarrow$ crystalline weight) is obviously true. So we would like to know what crystalline weights are automorphic.
Our previous paper \cite{Gao15} proved some results in this direction (in the unramified case), and this current paper is a continuation of the work in \textit{loc. cit.}.
Our paper gives the first evidence about weight part of Serre's conjectures where we can allow (degree 2) ramification at $p$, and dimension $d>2$. The method, similarly as in \cite{Gao15}, is a generalization of \cite{GLS14, GLS15}, but requires substantially more careful analysis of integral $p$-adic Hodge theory, see Subsection \ref{subsection: str of paper} for some highlights of the technical difficulties.

\subsubsection{Adapted bases of Kisin modules}
Now let us state more precisely our fist main local results. First, we set up some notations.
Let $p>2$, $K/\Qp$ a finite extension, $K_0$ the maximal unramified subfield, $\overline K$ a fixed algebraic closure, $G_K:=\Gal(\overline K/K)$ the absolute Galois group. Let $\pi$ be a fixed uniformizer of $K$, and $k$ the residue field. Let $f = [K_0 : \Q_p]$, $e = [K: K_0]$.
Let $E/\Qp$ (the coefficient field) be a finite extension that contains the image of all the embeddings $K \hookrightarrow \overline K$. Let $\O_E$ be the ring of integers, $\omega_E$ a fixed uniformizer, $k_E$ the residue field.
Fix an embedding $\kappa_0 \in \Hom_{\Q_p} (K_0,E)$, and recursively define $\kappa_i \in
\Hom_{\Qp}(K_0,E)$ for $i \in \Z$ so that $\kappa_{i+1}^p \equiv \kappa_i
\pmod{p}$.
Then $\Hom_{\Qp}(K_0,E)=\{\kappa_0,\ldots,\kappa_{f-1}\}$, with $\kappa_f = \kappa_0$.
Label $\Hom_{\Q_p}(K, E)$ as $\{\kappa_{ij}  :  0\leq i \leq f-1, 0 \leq j\leq e-1\}$ in any manner so that $\kappa_{ij}|_{K_0} =
\kappa_i$.

We will use the following listed notations often (see Subsection \ref{subsection_notations} for any unfamiliar terms):
\noindent (\textnormal{\textbf{CRYS}}). Let $p>2$ be an odd  prime, $K/\Qp$ a finite extension.
\begin{itemize}
  \item Let $V$ be a crystalline representation of $E$-dimension $d$, $D$ the filtered $\varphi$-module associated to $V$, which is a finite free $K_0 \otimes_{\Qp}E$-module.
 Let the labelled Hodge-Tate weights be $\HT_{\kappa_{i, j}}(D)=  \{0 = r_{i, j, 1} < \ldots < r_{i, j, d} \leq p\}$.
 We call $\HT_{\kappa_{i, 0}}(D)$ the \emph{principal} Hodge-Tate weights, and $\HT_{\kappa_{i, j}}(D), \forall j\neq 0$ the \emph{auxiliary} Hodge-Tate weights.

  \item Let $\rho=T$ be a $G_K$-stable $\mathcal O_E$-lattice in $V$, and $\hat \huaM \in \Mod_{\huaS_{\mathcal O_E}}^{\varphi, \Ghat}$ the $(\varphi, \Ghat)$-module attached to $T$. Let $\barrho: =T/\omega_ET$ be the reduction.
   \item Let $\hatm= \prod_{i=0}^{f-1} \hatm_i$ be the decomposition, where $\hatm_i=\varepsilon_i \hatm$. And similarly for the ambient Kisin module $\huaM =\prod_{i=0}^{f-1} \huaM_i$.
  \item Denote $\barhatm$ the reduction modulo $\omega_E$ of $\hat\huaM$, so it decomposes as $\barhatm =\prod_{i=0}^{f-1} \barhatm_i$. And similarly for the ambient Kisin module $\barm =\prod_{i=0}^{f-1} \barm_i$.
\end{itemize}

The following local result says that when the labeled Hodge-Tate weights are relatively small (a sort of ``$er \le p$ condition, if we use $r$ to denote the maximal Hodge-Tate weight), then the Kisin module admits an \emph{``adapted basis"}.

\begin{thm}\label{adpt basis}
With notations in (CRYS), suppose $p \nmid e$, and assume that $\sum_{j=0}^{e-1} r_{i,j,d} \leq p, \forall i$.
Then there exists a basis $\bolde_{i}$ of $\huaM_i$ such that,
$$\varphi(\bolde_{i-1}) = \bolde_{i} X_i \left(\prod_{j=1}^{e-1} \Lambda_{i,e-j} Z_{i,e-j} \right)
  \Lambda_{i,0} Y_i ,  $$
where
\begin{itemize}
  \item $X_i, Y_i, Z_{i, e-j} \in \GL_d(\O_E\llb u\rrb), \forall i, j$.
  \item $\overline Y_i:=Y_i  \pmod{\omega_E}  =Id$, and $\overline Z_{i, e-j}:= Z_{i, e-j}\pmod{\omega_E}  \in \GL_d(k_E)$.
  \item $\Lambda_{i, j}$ is the diagonal matrix $[(u-\pi_{ij})^{r_{i,j,x}}]_{x=1}^d$.
\end{itemize}
\end{thm}

Note that in the matrix of $\varphi$ with respect to the ``adapted basis" $\bolde_{i}$, we can \emph{``isolate"} the diagonal matrices $\Lambda_{i, j}$, where we can read off the labeled Hodge-Tate weights $r_{i,j,x}$ directly. (See Remark \ref{rem: adapted basis} for more comments on the significance of the adapted basis).
We remark that Theorem \ref{adpt basis} is the starting point, as well as a key ingredient in the proof of the following Theorem \ref{introlifting}.


\subsubsection{Crystalline liftings and Serre weight conjectures}
The following is our ``crystalline lifting" result.

\begin{thm} \label{introlifting}
With notations in (\textnormal{\textbf{CRYS}}). Suppose the ramification degree of $K$ is $e=2$. Suppose the Hodge-Tate weights of $V$ are such that $\HT_{\kappa _{i,0}} (V) =\{0 = r_{i,0, 1}< \ldots< r_{i,0, d} \leq p\}, \forall i$ and $\HT_{\kappa_{i, 1}} (V)=\{0, 1, \ldots, d-1\}, \forall i$.
Suppose that $\barrho$ is upper triangular. Suppose furthermore the following conditions on the Hodge-Tate weights are satisfied:
\begin{enumerate}
  \item (The principle weights are enough separated): $r_{i, 0, x+1} -r_{i, 0, x} \geq d, \forall i, x,$
  \item (The sum of weights is bounded): $r_{i, 0, d} +(d-1) \leq p-2, \forall i$.
\end{enumerate}
Then $\barrho$ has an upper triangular crystalline lift $\rho'$, such that $\HT_{i,j}(\rho')=\HT_{i,j}(\rho), \forall i, j$.
\end{thm}

Note that Theorem \ref{introlifting} is actually a special case of a slightly more general theorem (Theorem \ref{mainlocal}) that we prove. However, Theorem \ref{introlifting} is the case that can be applied to weight part of Serre's conjectures (see Remark \ref{rem:higher e}).


Our result in weight part of Serre's conjectures is straightforward application of the above crystalline lifting theorem (via automorphy lifting theorems of \cite{BLGGT14}), and we omit it in the introduction to save space, see Theorem \ref{application} for more detail.



\subsection{Structure of the paper.}\label{subsection: str of paper}
In Section \ref{section: filtration of kisin modules}, when $K/\Qp$ is tamely ramified and the Hodge-Tate weights satisfy certain conditions, we obtain a structure theorem for Kisin modules $\M$ associated to lattices in our crystalline representations.
The idea is very similar to \cite[\S 2]{GLS15}, namely, we study filtration structures of various modules.
However, we need to study \emph{``deeper"} level of filtrations (see Proposition \ref{existbase}), where the analysis is much more involved; in particular, many techniques in dimension $2$ (as in \cite{GLS15}) are no longer valid, e.g., see the remarks in the proof of Proposition \ref{above}.

In Section \ref{section: shape kisin}, we give a structure theorem for upper triangular reductions of crystalline representations that we study, which indeed gives us the \emph{upper bound} of all the possible shapes of upper triangular $\barm$ that we study.
The idea is to generalize results in \cite[\S 4]{Gao15}. However, the linear algebra is much more involved, and again similarly as in Section \ref{section: filtration of kisin modules}, many techniques in dimension $2$ are no longer valid. In particular, the structure result in Proposition \ref{shape} has never been observed before.

In Section \ref{section: final section lifting}, we prove our crystalline lifting theorem (using ideas and techniques from \cite{Gao15}), and apply it to weight part of Serre's conjectures.

\subsection{Notations.} \label{subsection_notations}
Many notations here are taken directly from \cite{GLS15, Gao15}, where the readers can find more details.

\subsubsection{Decomposition of rings} \label{subsection: decom of rings}
We have $W(k)\otimes_{\Zp}\O_E \simeq \prod_{i=0}^{f-1} W(k) \otimes_{ W(k), \kappa_i} \O_E,$
and suppose $1$ maps to $(\ve_0, \ldots, \ve_{f-1})$, then $\ve_i(W(k) \otimes_{\Zp} \O_E)
\simeq W(k) \otimes_{ W(k), \kappa_i} \O_E$.
Let $K_{0, E} =K_0\otimes_{\Qp}E$ and ${K}_E = K \otimes_{\Qp} E$. We have a natural decomposition $K_{0,E} = \prod
_{i=0}^{f-1} \ve_i ({K_{0, E }}) \simeq \prod _{i=0}^{f-1} E_{i}$
where $E_{i}= \ve_i ({K_{0, E}}) \simeq  K_0 \otimes_{K_0, \kappa_i}
E$. Similarly we have $K_E  = \prod_{i,j} E_{ij}$ with $E_{ij} = K
\otimes_{K , \kappa _{ij}} E$.  We will sometimes identify an element
$x \in E$ with an element of
$E_i$ via the map $x \mapsto 1 \otimes x$, and similarly for $E_{ij}$.

Let $\huaS: = W(k)\llb u\rrb$, $E(u)\in W(k)[u]$ the minimal polynomial of $\pi$ over $W(k)$, and $S$ the $p$-adic completion of the PD-envelope of $\huaS$ with respect to the ideal $(E(u))$. Let $A$ be an $\huaS$-algebra. We write $A_{\O_E} : = A \otimes_{\Z_p}
\O_E$, and $A_E := A \otimes_{\Z_p} E$. Then we have decompositions: $A_{\O_E} \simeq \prod_{i =0}^{f-1} A_{ \O_{E},i}$, with $A_{\O_{E},i} =\ve _i (A_{\O_E})  \simeq A  \otimes_{W(k), \kappa_i} \O_E,$ and
$A_E \simeq \prod_{i =0}^{f-1} A_{E,i}$, with $A_{E,i}=  A \otimes_{
  W(k),\kappa_i} E.$

We write $\iota$ for the isomorphism $\huaS_{\O_E}
\simeq   \prod_{i =0}^{f-1} \huaS_{ \O_{E},i}$ and $\iota_i$  the
projection  $\huaS_{\O_E} \twoheadrightarrow \huaS_{\O_{E},i} $. Then we have
$\huaS_{\O_{E},i} = \huaS \otimes_{W(k),\kappa_i} \O_E  \simeq
\O_{E_i}\llb u \rrb$, where $\O_{E_i}$ denotes the ring of integers in
$E_i$.
We write $\pi_{ij}= \kappa_{ij}(\pi) \in
E$. For each $\kappa_i$, we define $E^{\kappa_i} (u) =
\prod_{j=0}^{e-1} (u - \pi_{ij})$ in  $E[u]$,  so that
$E^{\kappa_i}(u)$ is just the polynomial obtained by acting on the coefficients of $E(u)$ by $\kappa_i$. Note that identifying $E_i$ with $E$ will
identify $\iota_i(E(u))$ with  $E^{\kappa_i} (u)$.

Let $f_\pi$ be the $W(k)$-linear map $S \to \O_K$ such that $u \mapsto \pi$.  We also denote $f_{\pi}$ for the map  $f_\pi \otimes_{\Z_p} E : S_ E
\to (\mathcal O_K)_E$.  We have surjections $\iota_{ij}:  (\mathcal O_K)_E \to K_E \to
E_{ij}$ for all $i,j$, and composing with $f_{\pi}$ gives $E$-linear
maps $f_{ij}:= \iota_{ij} \circ f_\pi: S_{E} \to E_{ij}$. Restricting the map
$f_{ij}$ to $\huaS_{\O_E}$ gives an $\O_E$-linear surjection $\huaS_{\mathcal O_E}
\to \O_{E_{ij}}$ which we also denote by $f_{ij}$ (Here $\O_{E_{ij}}$
denotes the ring of integers in $E_{ij}$).

\subsubsection{Decomposition of modules}

Recall that $D$ (the filtered $\varphi$-module associated to the crystalline representation $V$) is a finite free $K_{0, E}$-module of rank $d$, so that $D_K:= K \otimes_{K_0} D$ is a finite free $K_E$-module of rank $d$. We have natural decompositions $D = \oplus_{i =0}^{f-1} D_{i}$ with $D_i = \ve_i (D) \simeq D \otimes_{K_{0,E} }E _{i}$, and $D_K = \oplus_{i,j} D_{K, ij}$ with $D_{K,ij} = D_K \otimes_{K_E} E_{ij}$.
We always use $\{m_{ij}\}$ to denote a multi-set of integers where the indices $0 \le i \le f-1$ and $0 \le j \le e-1$.
For a given $\{m_{ij}\}$, we define
\[ \Fil ^{\{m_{ij}\}} D_K := \bigoplus_{i=0}^{f-1}
\bigoplus_{j=0}^{e-1} \Fil ^{m_{ij}}D_{K, ij}\subset D_K,
\]
where $\Fil ^{m_{ij}}D_{K, ij}$ is the filtration of $D_{K, ij}$ (with respect to the embedding $\kappa_{ij}: K \to E$).
If $m_{ij} = m$ for all $i,j$ then of course $\Fil ^{\{m_{ij}\}}D_K = \Fil ^m D_K$.

Let $\D := S \otimes_{W(k)} D$ be the Breuil module attached to
$D$, we have a natural isomorphism  $ D_K \simeq
\D \otimes_{S,f_{\pi}} \O_K$. We also have a natural
isomorphism  $D_{K, ij} \simeq \D \otimes_{S_{E}, f_{ij}} E_{ij}$. We
again denote  the projection  $\D \twoheadrightarrow D_K$ by $f_\pi$,
and the projection $\D \twoheadrightarrow D_K \twoheadrightarrow D_{K,
  ij}$ by $f_{ij}$.
For a given  $\{ m_{ij}\} $, we can inductively define a filtration $\Fil ^{\{m_{ij}\}} \D\subseteq \D$. We first set $\Fil ^{\{m_{ij}\}} \D = \D$ if $m_{ij} \leq 0$ for all $i,j$. Then define
\begin{equation*}
\Fil ^{\{m_{ij}\}} \D =\{x \in \D \, : \,f_{ij} (x) \in \Fil
^{m_{ij}}D_{K, ij}, \forall i,j, \text{ and }  N(x) \in \Fil
^{\{m_{ij}-1\}}\D    \}.
\end{equation*}
Note that evidently $\Fil ^m \D= \Fil
^{\{m_{ij}\}}\D$ when  $m_{ij} = m$ for all $i,j$.

Recall that  $ S _{E}\simeq \prod_{x
  =0}^{f-1}  S_{E,x}$, so $\D = \oplus_{x = 0}^{f-1} \D_x$
with $\D_x = \D \otimes_{S_{E}} S_{E,x}$.
We also have $\fijd =
\oplus_{x  =0} ^ {f-1}\Fil ^{\{m_{ij}\}} \D_{x  } $ with $\Fil
^{\{m_{ij}\}} \D_x = \fijd \otimes_{S_E} {S_{E,x}}$.
Note that  $ N(\D_x) \subset \D_{x}$ (because $N$ is $E$-linear on
$D$). So we see easily that $ \Fil ^{\{m_{i  j}\}}\D_x$ depends only on the $m_{x
  j}$ for $0 \le j \le e-1$, and  we can define
$ \Fil^{\{m_{x,0},\ldots,m_{x,e-1}\}} \D_{x} := \Fil ^{\{m_{i
    j}\}}\D_x$.  Note that $
\Fil^{\{m_{i,0},\ldots,m_{i,e-1}\}} \D_{i} $ also has the
recursive description
$$ \{x \in \D_i \, : \,
f_{i j} (x) \in \Fil ^{m_{i j}}D_{K, i j}, \forall j, \text{ and } N(x) \in \Fil^{\{m_{i,0}-1,\ldots,m_{i,e-1}-1\}}\D_i \}.$$

Recall that $\M$ (the Kisin module associated to a lattice in $V$) is an object in $\Mod_{\huaS_{\mathcal O_E}}^{\varphi}$,
which is a  finite free
$\huaS_{\O_E}$-module with rank $d = \dim_{E} V$, together with an $\O_E$-linear
$\varphi$-semilinear map $\varphi : \M \to \M$ such that the cokernel
of $1 \otimes \varphi : \huaS \otimes_{\varphi,\huaS} \M \to \M$ is killed
by $E(u)^r$. (We refer the readers to \cite[\S 1]{Gao15} for precise definitions and various properties of the categories $\Mod_{\huaS_{\mathcal O_E}}^{\varphi}$, $\Mod_{\huaS_{k_E}}^{\varphi}$, $\Mod_{\huaS_{\mathcal O_E}}^{\varphi, \Ghat}$, and $\Mod_{\huaS_{k_E}}^{\varphi, \Ghat}$.) Set $\M^* = \huaS \otimes_{\varphi, \huaS}\M$,
which can be viewed as a subset of $\D$.
Define
$\Fil ^{\{m_{ij}\}}\M^* := \M ^* \cap \fijd,$
and  $M_{K, ij}:= f_{ij}(\M^*) \subset D_{K, ij}$.  Similarly we
define $\M^*_i = \M^* \otimes_{\huaS_E} \huaS_{E,i}$ and
$\Fil^{\{m_{i,0},\ldots,m_{i,e-1}\}} \M^*_i := \M^*_i \cap \Fil^{\{m_{i,0},\ldots,m_{i,e-1}\}} \D_i.$

There is also filtration on $M_{K, ij}$ where $\Fil^m M_{K, ij}:=M_{K, ij} \cap \Fil^m D_{K, ij}$. We note that the sequence $\Fil^m M_{K, ij}$ is a sequence of finite free $\O_E$-modules, and it is a saturated sequence (i.e., the graded pieces are also finite free over $\O_E$).

\subsubsection{Some other notations.}
In this paper, we frequently use boldface letters (e.g., $\bolde$) to mean a sequence of objects (e.g., $\bolde=(e_1, \ldots, e_d)$ a basis of some module). We use $\Mat_d(?)$ to mean the set of matrices with elements in $?$, and $\GL_d(?)$ the set of invertible matrices. We use notations like $[u^{r_1}, \ldots, u^{r_d}]$ to mean a diagonal matrix with the diagonal elements in the bracket. We use $Id$ to mean the identity matrix. For a matrix $A$, we use $\diag A$ to mean the diagonal matrix formed by the diagonal of $A$.

In this paper, \textbf{upper triangular} always means successive extension of rank-$1$ objects.
We use notations like $\mathcal E(m_d, \dots, m_1)$ (note the order of objects) to mean the set of all upper triangular extensions of rank-1 objects in certain categories. That is, $m$ is in $\mathcal E(m_d, \dots, m_1)$ if there is an increasing filtration $0=\Fil^0 m \subset \Fil^1 m \subset \ldots \subset \Fil^d m =m$ such that $\Fil^i m /\Fil^{i-1}m =m_i, \forall 1 \leq i \leq d$.

We normalize the Hodge-Tate weights so that $\HT_{\kappa}(\varepsilon_p)={1}$ for any $\kappa: K \to \overline{K}$, where $\varepsilon_p$ is the $p$-adic cyclotomic character.

We fix a system of elements $\{\mu_{p^n}\}_{n=0}^{\infty}$ in $\barK$, where $\mu_1=1$, $\mu_p$ is a primitive $p$-th root of unity, and $\mu_{p^{n+1}}^p=\mu_{p^n}, \forall n$. We also fix a system of elements $\{\pi_n\}_{n=0}^{\infty}$ in $\barK$ where $\pi_0=\pi$ is the fixed uniformizer of $K$, and $\pi_{n+1}^p=\pi_n, \forall n$.
Let
$K_{p^{\infty}} = \cup_{n=0}^{\infty} K(\mu_{p^n})$, $\hat{K}=K_{\infty, p^{\infty}} = \cup_{n=0}^{\infty} K(\pi_n, \mu_{p^n}).$
Note that $\hat{K}$ is the Galois closure of $K_{\infty}$. Let
$\Ghat =\Gal(\hat{K}/K)$, $H_K= \Gal(\hat{K}/K_{\infty})$, and $G_{p^{\infty}} = \Gal(\hat{K}/K_{p^{\infty}}).$
When $p>2$, then $\hat G \simeq G_{p^{\infty}} \rtimes H_K$ and $G_{p^{\infty}} \simeq \Zp(1)$, and so we can (and do) fix a topological generator $\tau$ of $G_{p^{\infty}}$. And we can furthermore assume that $\mu_{p^n}=\frac{\tau(\pi_n)}{\pi_n}$ for all $n$.

\section{Kisin modules associated to crystalline representations} \label{section: filtration of kisin modules}

\newcommand{\e}{\mathfrak{e}}
\newcommand{\f}{\mathfrak{f}}
\newcommand{\boldfrake}{\boldsymbol{\mathfrak{e}}}

In this section, when $K/\Qp$ ($p>2$) is tamely ramified, we study the shape of Kisin modules $\huaM$, which are associated to $\O_E$-lattices in crystalline representations with certain technical restrictions on Hodge-Tate weights.
The method is to study the filtration structure of $\huaM^*$, which is a generalization of \cite[\S 2]{GLS15}.
In the second subsection, we obtain some results about rank-1 $(\varphi, \hat G)$-modules.
\subsection{Kisin modules: filtration structures and adapted bases}

\begin{defn}
Let $0 = r_1 < \dots< r_d$ be a sequence of integers. Let $M$ be a finite free module over a ring $R$, and let $\Fil^* M=\{\Fil^{r_x} M \}_{x=1}^d$ be a decreasing filtration of $M$ by finite free $R$-submodules. Let $m_1, \ldots, m_d$ be some elements in $M$, we say that the ordered sequence $\{m_x\}_{x=1}^d=\{m_1, \ldots, m_d\}$ \emph{fully generates} $\Fil^* M$ over $R$ if
\begin{itemize}
  \item $m_x \in \Fil^{r_x} M, \forall x$, and
  \item $\Fil^{r_x} M =\oplus_{y=x}^d Rm_y, \forall x$.
\end{itemize}
\end{defn}

\newcommand{\rr}{\mathfrak{r}}
\newcommand{\x}{_{x=1}^d}
\newcommand{\0}{\underline{0}}

\begin{notation} \label{notation: filtration superscript}
With notations in (CRYS), fix one $0 \leq i \leq f-1$ and fix one $0 \leq j \leq e-1$. We use $\rr$ to denote the sequence $r_{i,0,d}, \ldots, r_{i,j-1,d}$, and we use $\0$ to denote the sequence $0, \ldots, 0$ (with $e-1-j$ count of 0). This will simplify the filtration superscripts.
For example, we have
$\Fil^{\{\rr, n, \0\}} =\Fil^{\{r_{i,0,d}, \ldots, r_{i,j-1,d}, n, 0, \ldots, 0\}}$.
\end{notation}

\begin{prop}\label{fully_generates}
With notations in (CRYS), fix $i, j$ and use Notation \ref{notation: filtration superscript}. If there exists $(\e_{i,j,x}')_{x=1}^d$ such that
\begin{enumerate}
  \item $\e_{i,j,x}' \in  \Fil^{\{\rr, r_{i,j,x}, \0\}} \M_i^*= \Fil^{\{r_{i,0,d}, \ldots, r_{i,j-1,d}, r_{i,j,x},0,\ldots,0  \}}\M_i^*$;
  \item $\Fil^{\{\rr, 0, \0\}}\M_i^* =\oplus_{x=1}^d \huaS_{\O_E, i} \e_{i,j,x}'$;
  \item $\{f_{ij}(\e_{i,j,x}')\}_{x=1}^d$ fully generates $\{\Fil ^{r_{i,j, x}} D_{K, ij}\}_{x=1}^d$ over $E$.
\end{enumerate}
Then for any $n\geq r_{i,j, d}$ we have
$$\Fil^{\{\rr, n, \0\}}\M^*_i = \bigoplus_{x=1}^{d} \huaS_{\O_E,i} (u - \pi_{ij}) ^{n-r_{i, j,x}}\e_{i,j, x}'.$$
\end{prop}
\begin{proof}
This is a generalization of \cite[Prop. 2.3.3]{GLS15}, and the proof is similar, so we only give a sketch. However, we want to point out an \emph{important observation} here. In \textit{loc. cit.}, it would require (when $j=0$) that $\{f_{ij}(\e_{i,j,x}')\}_{x=1}^d$ to fully generate the filtration of $\Fil^* M_{K,ij}$ over $\O_E$. However, it is not necessary at all. All we need is that $\{f_{ij}(\e_{i,j,x}')\}_{x=1}^d$ fully generates the filtration of $\Fil^* D_{K,ij}$ over $E$.

Firstly, we prove by induction that for any $0 \leq n \leq r_{i,j,d}$, we have
\numequation
\label{Ddecomp}
\Fil^{\{\rr, n, \0\}}\D_i = \bigoplus_{x=1}^{d} \huaS_{\O_E,i} (u - \pi_{ij}) ^{\lceil n-r_{i,j, x}\rceil}\e_{i,j,x}'+ (\Fil^p S_{E,i}) \D_i.
\end{equation}
Here, the notation $\lceil a \rceil$ is such that $\lceil a \rceil=a$ if $a\geq 0$ and $\lceil a \rceil=0$ if $a<0$.

The case $n=0$ is trivial. Suppose the Statement \eqref{Ddecomp} is valid for $n-1$, and consider it for $n$. Now similarly as in \cite[Prop. 2.3.3]{GLS15}, we set:
\begin{itemize}[leftmargin=*]
  \item $\widetilde \Fil^{\{\rr, n, \0\}}\D_i :=$ right hand side of \eqref{Ddecomp}, and
  \item $\widetilde \Fil ^{\{m_0, \ldots, m_j,\0\}}\D_i:=\Fil ^{\{m_0, \ldots, m_j, \0\}}\D_i$ for $\{m_0, \ldots, m_j, \0\}< \{\rr, n, \0\}$ in dictionary order.
\end{itemize}

We can easily check that $\widetilde \Fil$ satisfies all the conditions of  \cite[Prop. 2.1.12]{GLS15}, in particular, we can check that
\begin{itemize}[leftmargin=*]
  \item $N(\widetilde \Fil ^{\{m_0, \ldots, m_j, \0\}}\D_i) \subseteq \widetilde \Fil ^{\{m_0-1, \ldots, m_j-1,\0\}}\D_i$, because $\e_{i,j,x}' \in \Fil^{\{\rr, r_{i,j,x}, \0\}}\M_i^*$.

  \item $f_{\pi}(\widetilde \Fil^{\{\rr, n, \0\}}\D_i) =\left( \oplus_{y=0}^{j-1} \Fil^{r_{i,y,d}}D_{K, iy} \right)\oplus \Fil^{n}D_{K, ij} \oplus \left( \oplus_{y=j+1}^d  \Fil^0 D_{K, iy} \right)$, by using the assumption that $\{f_{ij} (\e_{i,j, x}')\}_{x=1}^d$ fully generates $\{\Fil ^{r_{i, j,x}} D_{K, ij}\}_{x=1}^d$.

  \item $E^{\kappa_i}(u)\widetilde \Fil ^{\{\rr-1, n-1, \0\}}\D_i \subseteq \Fil ^{\{\rr, n, \0\}}\D_i$. Indeed, we have
  \begin{eqnarray*}
        &&E^{\kappa_i}(u)\widetilde \Fil ^{\{\rr-1, n-1, \0\}}\D_i \\ &=&\left(\prod_{y=j+1}^{e-1}(u-\pi_{iy})\right)(u-\pi_{ij})\left(\prod_{y=0}^{j-1}(u-\pi_{iy})\right)\widetilde \Fil ^{\{\rr-1, n-1, \0 \}}\D_i \\
         &\subseteq&  (u-\pi_{ij})\widetilde \Fil ^{\{\rr,n-1,\0 \}}\D_i\\
        &\subseteq&   \widetilde \Fil ^{\{\rr,n,\0 \}}\D_i (\textnormal{by using \eqref{Ddecomp} for } n-1).
      \end{eqnarray*}
\end{itemize}
So now by \cite[Prop. 2.1.12]{GLS15}, we must have
$\widetilde \Fil^{\{\rr,n,\0 \}}\D_i=\Fil^{\{\rr,n,\0 \}}\D_i$, and so \eqref{Ddecomp} is proved.

Now, by \cite[Lem. 2.2.1(3)]{GLS15}, to prove our proposition, it suffices to prove it for $n=r_{i,j,d}$. This can be easily achieved by the similar argument as the last paragraph of \cite[Prop. 2.3.3]{GLS15}. (Note that there is a minor error in \textit{loc. cit.}, namely, $\Fil^r S_{E, i} \cap \huaS_{\O_{E}, i}$ should be equal to $\left(E^{\kappa_i}(u)\right)^r \huaS_{\O_{E}, i}$, not $(u-\pi_{i0})^{r}\huaS_{\O_{E}, i}$).
\end{proof}

Now we \emph{verify} the assumptions in the above proposition. First, we state the $j=0$ case in Proposition \ref{level0base} (with something extra), and then the general case in Proposition \ref{existbase}.
\begin{prop}\label{level0base} With notations in (CRYS).
\begin{enumerate}
  \item There exists a basis $(e_{i-1, 0, x})_{x=1}^d$ of $\M_{i-1}$, such that $\{f_{i, 0} (e_{i-1, 0, x}) \}_{x=1}^d$ fully generates $\{\Fil^{r_{i, 0, x}}M_{K, i, 0}\}_{x=1}^d$ over $\O_E$.
  \item There exists $(\e_{i,0,x}')_{x=1}^d$ in $\M^*_i$, such that
  \begin{enumerate}
    \item $\e_{i,0,x}'-e_{i-1, 0, x} \in \omega_E \M^*_i, \forall x$.
    \item  $\e_{i,0,x}' \in \Fil^{\{r_{i, 0, x}, 0, \ldots, 0\}} \M^*_i, \forall x$, and $(\e_{i,0,x}')_{x=1}^d$ form a basis of $\M^*_i$.
    \item $\{f_{i, 0}(\e_{i,0,x}')\}_{x=1}^d$ fully generates $\{\Fil^{r_{i, 0, x}}D_{K, i 0}\}_{x=1}^d$ over $E$.
  \end{enumerate}
\end{enumerate}
\end{prop}
\begin{proof}
This is easy generalization (from dimension 2 to higher dimension) of \cite[Prop. 2.3.5]{GLS15}. Note that Item (2)(c) is actually an easy consequence of (2)(a). We still list it just to emphasize it.
\end{proof}

Before we state our next proposition, we introduce the following definition and lemma.

\begin{defn} \label{defn: property B}
Let $\pi$ be an element in the maximal ideal of $\O_E$, and $v_{\pi}$ the valuation of $E$ such that $v_{\pi}(\pi)=1$.
A polynomial $f(u) \in E[u]$ is said to satisfy (Property B) with respect to $v_{\pi}$, if when written as $f(u) = \sum_{i=0}^N a_i(u-\pi)^i$, we have $v_{\pi}(a_i) \geq -i, \forall i$.
\end{defn}

\begin{lemma}\label{lemma: property B} \hfill
\begin{enumerate}
  \item If $f_1, f_2 \in E[u]$ satisfy (Property B), then so is $f_1f_2$.
  \item With notations in Subsection \ref{subsection: decom of rings}, suppose $p\nmid e$ (i.e., $K$ is tamely ramified). Let $0 \leq q <j \leq e-1$, then $\frac{u-\pi_{iq}}{\pi_{ij}-\pi_{iq}}$ satisfies (Property B) with respect to $v_{\pi_{ij}}$.
\end{enumerate}
\end{lemma}
\begin{proof}
(1) is easy.
To prove (2), first consider $E^{\kappa_i}(u)=\prod_{q=0}^{e-1}(u-\pi_{iq}) =u^e +pF(u)$. Take derivative, evaluate at $\pi_{ij}$, and consider $v_{\pi_{ij}}$ on both sides, then one can easily see that $v_{\pi_{ij}}(\pi_{ij} -\pi_{iq})=1, \forall q\neq j$ (using that $p \nmid e$). Since $\frac{u-\pi_{iq}}{\pi_{ij}-\pi_{iq}} =1+  \frac{u-\pi_{ij}}{\pi_{ij}-\pi_{iq}} $, we can conclude (2).
\end{proof}

\begin{prop}\label{existbase}
With notations in (CRYS), fix $i, j$ and use Notation \ref{notation: filtration superscript}.
Suppose $p \nmid e$ (i.e., $K$ is tamely ramified), and assume that
\begin{itemize}
  \item (A0): For each $i$, $\sum_{q=0}^{e-1} r_{i, q, d} \leq p$.
\end{itemize}
Then there exists $(\e_{i,j,x}')_{x=1}^d$ such that
\begin{enumerate}
  \item $\e_{i,j,x}' \in \Fil^{\{\rr, r_{i,j,x},\0  \}}\M_i^*$, and $\Fil^{\{\rr, 0,\0  \}}\M_i^*=\oplus_{x=1}^d \huaS_{\O_{E}, i}\e_{i,j,x}'$.
  \item $\{f_{ij}(\e_{i,j,x}')\}_{x=1}^d$ fully generates $\{\Fil ^{r_{i,j, x}} D_{K, ij}\}_{x=1}^d$ over $E$.
  \item $\Fil^{\{\rr, r_{i,j,d},\0  \}}\M_i^* =\oplus_{x=1}^d  \huaS_{\O_{E}, i}\alpha_{i, j, x}'$ where $\alpha_{i, j, x}'= (u-\pi_{ij})^{r_{i, j, d}-r_{i, j, x}}\e_{i, j, x}'$.
  \item We have $(\e'_{i, j, x})_{x=1}^d=(\e'_{i, j-1, x})_{x=1}^d [(u-\pi_{i, j-1})^{r_{i, j-1, d}-r_{i, j-1, x}}]_{x=1}^d Z'_{i, j}$, where $Z'_{i, j}=B_{i, j}A_{i, j}$ for some $B_{i, j} \in \GL_d(\O_E)$ and $A_{i,j} \in \GL_d(\O_E\llb u \rrb)$ such that $A_{i, j}\equiv Id\pmod{ \omega_E}$.
\end{enumerate}
\end{prop}
\begin{proof}
We prove it by induction on $j$. When $j=0$, this is precisely Proposition \ref{level0base}(2).
Now suppose our proposition is true for $j-1$, and now consider it for $j$. Since we are working with a fixed $i$, so we drop $i$ from all the subscripts.

\textbf{Step 0}.
Let $L_q:  = [(u-\pi_q)^{r_{q, d} -r_{q, x} }  ]_{x=1}^d, \forall 0 \leq q \leq e-1$ be the diagonal matrices, then we have
$$ (\alpha'_{j-1, x})_{x=1}^d  =  (\e_{j-1, x}')_{x=1}^d  L_{j-1} =  ( \e_{0, x}')_{x=1}^d  \left( \prod_{q=0}^{j-2} (L_qZ_{q+1}') \right) L_{j-1} .$$
Note that the $\O_E$-linear span of $\{f_j(\e'_{0, x})\}\x$ is equal to $f_j(\M^*_i)=M_{K, ij}$, which is a lattice in $D_{K, j}$.
Now $f_j (u -\pi_q)=\pi_j -\pi_q \in \O_E, \forall 0 \leq q \leq j-1$, and $Z'_{q+1} \in \GL_d(\O_E\llb u \rrb), \forall 0 \leq q \leq j-2$,
so it is easy to see that the $\O_E$-linear span of $\{f_j(\alpha'_{j-1, x})\}\x$
is also a lattice in $D_{K, j}$.
Now, by \cite[Lem. 4.4]{GLS14}, we can find some $B_j\in \GL_d(\O_E)$ such that if we let $ (\f_{j-1, x})\x  := (\alpha_{j-1, x}')\x B_j  $, then we have that
\numequation \label{fullgen}
 \{ f_j( \f_{j-1, x} ) \}\x \textnormal{ fully generates } \{  \Fil^{r_{j, x}} D_{K, j}  \}\x \textnormal{over } E.
\end{equation}

\textbf{Step 1}. For each $x$, if $r_{j, x}=0$, then we let $\f_{j, x}^{(0)}:=\f_{j-1, x}$. If $r_{j, x}>0$, then for every $1 \leq n \leq r_{j, x}$, we want to construct
$$\f_{j, x}^{(n)} = \f_{j-1, x} +\left( \pi_j \prod_{q=0}^{j-1}(u-\pi_q)^{r_{q, d}} \right) \sum_{s=1}^{n-1} (u-\pi_j)^s \left(\sum_{y=1}^d a_{s, x, y}^{(n)}\e'_{i, 0, y} \right) \in
\Fil^{\{\rr, n,\0\}}\M^*_i,$$
with some $a_{s, x, y}^{(n)} \in \O_E$.
Note that we have kept $i$ in the subscript of $\e'_{i, 0, y}$ to emphasize it.
Let us now fix $x$, and drop $x$ from the subscripts. We denote $\f^{(n)}:=\f_{j, x}^{(n)}$ and $\f:=\f_{j-1, x}$ in order for easier comparison with the proof of \cite[Prop. 2.3.5]{GLS15}. We also denote $\f_y:=\e'_{i, 0, y}$, $Q:=Q_{j-1}= \prod_{q=0}^{j-1}(u-\pi_q)^{r_{q, d}}$ for brevity. That is, we want to construct
\numequation
\label{construct}
\f^{(n)} = \f +  \pi_j Q  \sum_{s=1}^{n-1}  (u-\pi_j)^s \left(\sum_{y=1}^d a_{s, y}^{(n)}\f_{y}\right) \in \Fil^{\{\rr, n,\0\}}\M^*_i.
\end{equation}

We prove \eqref{construct} by induction on $n$. When $n=1$, it suffices to show that $\f=\f_{j-1, x}  \in \Fil^{\{\rr, 1,\0\}}\M^*_i$, which can be concluded from the following:
\begin{itemize}
\item  $N(\f) \in \Fil^{\{\rr-1, 0, \0\}}\M^*_i$, because $\f \in \Fil^{\{\rr, 0, \0\}}\M^*_i$.
\item  $ f_q (\f)  \in \Fil^{r_{q, d}} D_{K, q}, \forall 0 \leq q \leq j-1  $ because $\f \in \Fil^{\{\rr, 0, \0\}}\M^*_i$; and $f_j (\f)  \in \Fil^{r_{j, x}} D_{K, j} \subset \Fil^1 D_{K, j}$ because of \eqref{fullgen}.
\end{itemize}
Suppose we have constructed \eqref{construct} for $n$, and now consider $n+1$ (which is $\leq r_{j, x}$).

Let
\numequation
\label{Hpoly}
H:=H_j(u)=  \frac{(u-\pi_j)}{\pi_j} \frac{\prod_{q=0}^{j-1} (u-\pi_q)^{ \max\{r_{q, d}, 1\} }  } { \prod_{q=0}^{j-1}(\pi_j-\pi_q)^{\max\{r_{q, d}, 1\}} }.
\end{equation}
The polynomial $H$ satisfies the following properties:
\begin{itemize}
  \item (Property A): $(u-\pi_j) \mid  N(H(u))+1$, where $N: S[1/p] \to S[1/p]$ the $K_0$-linear differential operator such that $N(u)=-u$.
  \item (Property B) with respect to $v_{\pi_j}$ as in Definition \ref{defn: property B}, by Lemma \ref{lemma: property B}.
 \end{itemize}
Now, we consider
\numequation
\label{ftilde}
\widetilde\f^{(n+1)}: = \sum _{\ell  =0} ^ n \frac{H(u)^\ell N^\ell (\f^{(n)})}{\ell !}.
\end{equation}
Then we have
$$ N(\tilde\f^{(n+1)}) = \frac{H(u)^n}{n!} N^{n+1}(\f^{(n)}) + \sum_{\ell=1}^{n}
\frac{(1+N(H(u)))H(u)^{\ell-1} N^{\ell}(\f^{(n)})}{(\ell-1)!}.$$
One can easily see that
\begin{itemize}[leftmargin=*]
  \item $N(\tilde\f^{(n+1)}) \in \Fil^{\{\rr-1, n, \0\}} \D_i$, using (Property A) of $H(u)$ above.

  \item When $0 \leq q \leq j-1$, $f_{q}(\tilde\f^{(n+1)})=f_{q}(\f^{(n)})$ (since $H(\pi_q)=0$), which is in $\Fil^{r_{q, d}}D_{K, q}$, because $\f^{(n)} \in \Fil^{\{\rr, n, \0\}}\M^*_i$ by induction hypothesis.

  \item $f_{j}(\tilde\f^{(n+1)})=f_{j}(\f^{(n)})=f_{j}(\f) \in \Fil^{r_{j, x}}D_{K, j} \subseteq \Fil^{n+1}D_{K, i1}$ because $n+1\leq r_{j, x}$.

\end{itemize}
So we can conclude that $\tilde\f^{(n+1)} \in \Fil^{\{\rr, n+1, \0\}} \D_i$.

We have
\nummultline
\label{bigeq}
\widetilde \f^{(n+1)} -  \f^{(n)} =   \sum _{\ell  =1} ^ n \frac{ (H)^\ell} { \ell!} N^\ell
\left(\f +  \pi_j Q \sum_{s=1}^{n-1} (u-\pi_j)^s\left  (\sum_{y=1}^d a_{s, y}^{(n)}\f_{y}\right )\right )\\
= \sum _{\ell  =1} ^ n \frac{ (H)^\ell} { \ell!}
N^\ell (\f) +
  \sum_{\ell  =1}^n \sum_{s=1}^{n-1} \sum_{t=0}^{\ell} \frac{(H)^\ell}{\ell!} \binom{\ell}{t}
  N^{\ell -t} \Bigg( \pi_j Q(u -\pi_j)^s \Bigg)
  \left  (\sum_{y=1}^d a_{s, y}^{(n)}N^t(\f_{y}) \right ).
\end{multline}
Denote the right hand side of \eqref{bigeq} as $F$. Let $G=: \frac{Q}{ \prod_{q=0}^{j-1}(\pi_j-\pi_q)^{r_{q, d}}}$, then $\frac{H}{G}$ is still in $E[u]$.
Factor $G$ from $F$, and write $F=G\frac{F}{G}$.
For $\beta=\f$ or one of $\f_y, 1\leq y \leq d$, let us write
\numequation
\label{subst}
\D_i \ni N^{t}(\beta) = \sum_{\tilde{m}=0}^{\infty}  (u-\pi_j)^{\tilde m} (\sum_{z=1}^d  c_{\tilde{m}, z}^t(\beta) \f_{z}), \textnormal{ where }c_{\tilde{m}, z}^t(\beta) \in E.
\end{equation}
Substituting \eqref{subst} into $\frac{F}{G}$, and rewrite all element (e.g., $\frac{H^{\ell}}{G\ell!}$, $N^{\ell -t}( \pi_j Q(u -\pi_j)^s )$) in $E[u]$ to be in the ring $E[u-\pi_j]$.
Then we can collect terms and write
$$\widetilde \f^{(n+1)} = \f^{(n)} + G \left(\sum_{m=1}^{\infty} (u-\pi_j)^m (\sum_{y=1}^d b_{m,y}\f_y)\right), \textnormal{ where } b_{m,y}\in E.$$
And finally, we can define
$$\f^{(n+1)}:= \f^{(n)}+ G\left( \sum_{m=1}^{n} (u-\pi_j)^m (\sum_{y=1}^d b_{m,y}\f_y)\right).$$
Since the terms that we throw away is $G \left(\sum_{m=n+1}^{\infty} (u-\pi_j)^m (\sum_{y=1}^d b_{m,y}\f_y)\right)$, which belongs in $\Fil^{\{\rr, n+1, \0\}}\D_i$, so we have $\f^{(n+1)} \in \Fil^{\{\rr, n+1, \0\}}\D_i$.

\textbf{Step 2}.
In order to finish our construction of \eqref{construct} for $n+1$, it suffices to show that for each $1\leq m \leq n$,
$v_{\pi_j}(b_{m,y}) \geq 1+ |\rr|$ where $|\rr|:=\sum_{q=0}^{j-1} r_{q, d}$ (Note that the denominator of $G$ has $v_{\pi_j}$ equal to $|\rr|$).
First, we list the following facts in controlling the coefficients.
\begin{itemize}[leftmargin=*]
  \item (Coe-1): Let $S'=W(k)\llb u^p, \frac{u^{ep}}{p}\rrb[\frac{1}{p}] \cap S$ and set $\mathcal I_{\ell} = \sum_{m=1}^{\ell} p^{\ell-m} u^{pm} S'$.
      Fix one $1\leq \ell \leq p$,
      then for any $\beta \in \huaM_i^*$, $N^{\ell}(\beta) = \sum_y w_y \f_y $ with $w_y \in \mathcal I_{\ell}$ by \cite[Cor. 4.11]{GLS14}.
      For each $w$ of $w_y$, expand it as $w=\sum_{i=0}^{\infty} a_i(u-\pi_j)^i$, then by \cite[Lem. 2.3.9]{GLS15}, we have
      $v_{\pi_j} (a_0) \geq p+(\ell-1)\min\{p, e\}$ and $v_{\pi_j} (a_i) \geq p+e-i+(\ell -1)\min\{p, e\}, \forall 1 \leq i \leq p-1$. But we only need a weak form of these results, namely,
    \begin{equation*}
    v_{\pi_j} (a_i) \geq p-i, \forall 0 \leq i \leq p-1.
\end{equation*}

\item (Coe-2): When $1\leq \ell <p$, the polynomial $\frac{(H)^\ell} {G\ell!}$ satisfies (Property B) of \ref{defn: property B} with respect to $v_{\pi_j}$.

\end{itemize}

Now we analyze the terms in $\frac{F}{G}$, and prove that for any $m\leq n$, each $(u-\pi_j)^m$ appearing in it has coefficient divisible by $\pi_j^{1+ |\rr|}$.

\textbf{\underline{\emph{(Part 1)}}}. For the part $\sum _{\ell  =1} ^ n \frac{(H)^\ell} {G\ell!} N^\ell (\f)$. To build a term with $(u-\pi_j)^m$ where $1\leq m \leq n$, suppose that $H^{\ell}/G$ contributes $(u-\pi_j)^{m_1}$, and in $N^\ell (\f)$, some $w$ (as in (Coe-1) above) contributes $(u-\pi_j)^{m_2}$. Then using (Coe-2) and (Coe-1), the coefficient of this $(u-\pi_j)^{m_1+m_2}$ has $v_{\pi_j}$ at least $-m_1+p-m_2 \geq p-n \geq p-(r_{j, d}-1)$, which is $\geq |\rr|+1$ by Assumption (A0).

\textbf{\underline{\emph{(Part 2)}}}. For the part of $$\frac{  (H)^\ell}{G\ell!} \binom{\ell}{t}  N^{\ell -t} \Bigg(\pi_j Q(u -\pi_j)^s\Bigg) \left  (\sum_{y=1}^d a_{s, y}^{(n)} N^{t}(\f_{y}) \right ),$$
it suffices to consider the $\pi_j$-powers in the coefficients of
$$    \frac{(H)^\ell}{G}   N^{\ell -t} \left(\pi_j Q(u -\pi_j)^s\right)   N^{t}(\f_{y}), $$
because $a_{s, y}^{(n)} \in \O_E$ by induction hypothesis, and $\ell<p$.
To build a term with $(u-\pi_j)^m$ where $1\leq m \leq n$, suppose that $H^{\ell}/G$ contributes $(u-\pi_j)^{m_1}$, $N^{\ell -t} \left (\pi_j Q(u -\pi_j)^s\right )$ contributes $(u-\pi_j)^{m_2}$, and $N^{t}(\f_{y})$ contributes $(u-\pi_j)^{m_3}$.
Using (Coe-2) and (Coe-1), and note that $N^{\ell -t} \left (\pi_j Q(u -\pi_j)^s\right )$ can be written as a polynomial in $\O_E[u-\pi_j]$, we can see that the coefficient of this $(u-\pi_j)^{m_1+m_2+m_3}$ has $v_{\pi_j}$ at least $-m_1+p-m_3 \geq p-n \geq p-(r_{j, d}-1)$, which is $\geq |\rr|+1$ by Assumption (A0).

Now we can conclude that \eqref{construct} is valid for all $1 \leq n \leq r_{j, x}$.

\textbf{Step 3.} Finally, we can define
$$(\e'_{j, x})\x : = (\f_{j, x}^{(r_{j, x})})\x,$$
and we can conclude the proof of our proposition, by verifying that $ (\e'_{j, x})\x$ satisfies all the four items in our proposition:
\begin{itemize}[leftmargin=*]
 \item For Item (4). Let us write \eqref{construct} as
 $$\f^{(n)} =\f+ \pi_j P.$$
 Since both $\f^{(n)}$ and $\f$ are in $\Fil^{\{\rr, 0,\0\}}\M^*_i$, so $\pi_j P \in \Fil^{\{\rr, 0,\0\}}\M^*_i$, and so $P \in \Fil^{\{\rr, 0,\0\}}\M^*_i$ too.
     This means that $P$ can be written as a $\huaS_{\O_E, i}$-linear combination of the basis $(\f_{j-1, x})\x$ of $\Fil^{\{\rr, 0,\0\}}\M^*_i$.
 Then it is clear from \eqref{construct} that $(\e'_{j, x})\x = (\f_{j-1, x})\x A_{j}$ for some matrix $A_{j} \in \GL_d(\O_E\llb u \rrb)$ such that $A_{j} \equiv Id \pmod{\omega_E}$.
 \item For Item (1). Apply induction hypothesis to Item (3), $(\alpha'_{j-1, x})\x$ is a basis for $\Fil^{\{\rr, 0, \0\}}$, so $(\e'_{j, x})\x = (\f_{j-1, x})\x A_j = (\alpha'_{j-1, x})\x B_jA_j$ is also a basis.
 \item For Item (2). By \eqref{construct}, we have $f_j(\e'_{j, x}) = f_j (\f_{j-1, x})$, and then apply \eqref{fullgen}.
 \item Item (3) is a consequence of Item (1) and (2), by Proposition \ref{fully_generates}.
\end{itemize}
\end{proof}

\begin{corollary}\label{above}
With notations in (CRYS), suppose $p \nmid e$, and assume that $\sum_{j=0}^{e-1} r_{i,j,d} \leq p, \forall i$.
Then we have $\Fil ^{\{p, p, \dots, p\}}\M^*_i =\bigoplus_{x=1}^d \huaS_{\O_E,i} \alpha_{i, e-1, x},$
where
$$(\alpha_{i, e-1, 1}, \ldots, \alpha_{i, e-1, d})= (\e'_{i,0, 1}, \ldots, \e'_{i,0, d}) \Lambda'_{i,0} \left(\prod_{j =1}^{e-1} Z'_{ij} \Lambda'_{ij} \right)  $$
such that
\begin{itemize}
  \item $\e'_{i,0, x}$ as in Proposition \ref{level0base}(2).
  \item $\Lambda'_ {ij}= [(u - \pi _{i,j})^{p-r_{ij, 1}}, \ldots, (u - \pi _{i,j})^{p-r_{ij, d}}], \forall 0 \leq j \leq e-1$.
  \item $Z'_{ij} =B_{i,j}A_{i,j} \in \GL_d (\O_E\llb u\rrb)$ as in Proposition \ref{existbase}(4).
\end{itemize}
\end{corollary}
\begin{proof}
This is easy corollary of Proposition \ref{existbase}. Note that \cite[Cor. 2.3.10]{GLS15} proved the case for $d=2$ where all the auxiliary labelled Hodge-Tate weights are $\{0, 1\}$. But the argument of \textit{loc. cit.} relies on this special shape of auxiliary labelled Hodge-Tate weights, and cannot be generalized.
\end{proof}

\begin{thm}\label{shapecorollary}
With notations in (CRYS), suppose $p \nmid e$, and assume that $\sum_{j=0}^{e-1} r_{i,j,d} \leq p, \forall i$.
Then there exists a basis $\bolde_{i}$ of $\huaM_i$ such that,
$$\varphi(\bolde_{i-1}) = \bolde_{i} X_i \left(\prod_{j=1}^{e-1} \Lambda_{i,e-j} Z_{i,e-j} \right)
  \Lambda_{i,0} Y_i ,  $$
where
\begin{itemize}
  \item $X_i, Y_i, Z_{i, e-j} \in \GL_d(\O_E\llb u\rrb), \forall i, j$.
  \item $\overline Y_i:=Y_i  \pmod{\omega_E}  =Id$, and $\overline Z_{i, e-j}:= Z_{i, e-j}\pmod{\omega_E}  \in \GL_d(k_E)$.
  \item $\Lambda_{i, j}$ is the diagonal matrix $[(u-\pi_{ij})^{r_{i,j,x}}]_{x=1}^d$.
\end{itemize}
\end{thm}
\begin{proof}
Similar as in \cite[Thm. 2.4.1]{GLS15}. Note that in \loccit, $Z_{i, e-j} \in \GL_d(\O_E)$. But here in our situation, we will only have $Z_{i, e-j} \in \GL_d(\O_E\llb u\rrb)$. Fortunately, we still have the reduction $\overline Z_{i, e-j} \in \GL_d(k_E)$, because $Z_{i, e-j} =(Z_{i, e-j}')^{-1}$ and $\overline Z_{i, e-j}' =\overline B_{i, e-j} \overline A_{i, e-j} \in \GL_d(k_E)$ by Proposition \ref{existbase}(4).
\end{proof}

\begin{remark}\label{rem: adapted basis}
\begin{enumerate}
  \item We call the basis $\bolde_{i}$ ($0 \le i \le f-1$) the ``adapted basis" for $\varphi$. In fact, we usually call those elements $\alpha_{i, e-1, x}$ in Corollary \ref{above} the ``base adapt\'ee" for the filtration $\Fil ^{\{p, p, \dots, p\}}\M^*_i$, see \cite[Def. 2.2.1.4]{Bre99a}. So we are slightly abusing the terminology here, although it is easy to see that they determine each other.
  \item (The following remark is inspired from discussions with Tong Liu.)
   As we already mentioned after Theorem \ref{adpt basis} in the Introduction, the adapted basis allows us to isolate the diagonal matrices $\Lambda_{i, j}$. The significance of the ``isolation of $\Lambda_{i, j}$" is that when we consider crystalline liftings (of some residual representation) with \emph{fixed} labeled Hodge-Tate weights, then the corresponding Kisin modules will always admit adapted bases, and the only part (of the matrix for $\varphi$) that is changing (i.e., deforming) are the matrices $X_i, Y_i, Z_{i, e-j}$.

      The adapted bases help to study reductions of crystalline representations (see Proposition \ref{tameinertia}), which in turn are needed when considering crystalline lifting problems. Indeed for example, the key fact that the argument of \cite[Lem. 1.4.2]{BLGGT14} can work is the existence of adapted basis in the Fontaine-Laffaille case (which is much easier than our situation). We expect our adapted bases of Kisin modules can have some similar applications in the future.
\end{enumerate}
\end{remark}

\subsection{Rank-1 $(\varphi, \hat G)$-modules in the tamely ramified case}

\begin{defn}
Let $\boldt = (t_0, \ldots, t_{f-1})$ be a sequence of non-negative integers, $a \in k_{E}^{\times}$. Let $\barm(\boldt; a):=\barm(t_0, \ldots, t_{f-1}; a)= \prod_{i=0}^{f-1}\barm(\boldt; a)_i$ be the rank-$1$ module in $\Mod_{\huaS_{k_E}}^{\varphi}$ such that
\begin{itemize}
\item $\barm(\boldt; a)_i$ is generated by $e_i$, and
\item $\varphi(e_{i-1})=(a)_i u^{t_i}e_i$, where $(a)_i=a$ if $i=0$ and $(a)_i=1$ otherwise.
\end{itemize}
\end{defn}

\begin{defn} \label{huaMrank1}
Let $(r_{i, j})$ be an $f\times e$-matrix where $r_{i, j} \geq 0$ are integers.
Let $\hat a \in \O_E^{\times}$.
Let $\huaM((r_{i, j}); \hat a):= \prod_{i=0}^{f-1}\huaM((r_{i, j});  \hat a)_i$ be the rank-$1$ module in $\Mod_{\huaS_{\mathcal O_E}}^{\varphi}$ such that
\begin{itemize}
\item $\huaM((r_{i, j});  \hat a)_i$ is generated by $\hat e_i$, and
\item $\varphi(\hat e_{i-1})=(\hat a)_i \prod_{j=0}^{e-1}(u-\pi_{ij})^{r_{i, j}}\hat e_i$, where $(\hat a)_0=\hat a$ and $(\hat a)_i=1, \forall i\neq 0$.
\end{itemize}
\end{defn}

\begin{lemma}\label{rank1} \hfill
 \begin{enumerate}
   \item If $a \in k_E^{\times}$ is the reduction of $\hat a$, then $\huaM((r_{i, j}); \hat a)/\omega_E\huaM((r_{i, j}); \hat a) \simeq \barm(\boldt; a)$, where $\boldt = (t_0, \ldots, t_{f-1})$ with $t_i=\sum_{j=0}^{e-1}r_{i, j}$.

   \item When $p \nmid e$, there is a unique $\hat \huaM((r_{i, j}); \hat a) \in \Mod_{\huaS_{\mathcal O_E}}^{\varphi, \Ghat}$ such that
    \begin{itemize}
      \item The ambient Kisin module of $\hat \huaM((r_{i, j}); \hat a)$ is $\huaM((r_{i, j}); \hat a)$, and
      \item $\That(\hat \huaM((r_{i, j}); \hat a))$ is a crystalline character.
    \end{itemize}
   In fact,
    $\That(\hat \huaM((r_{i, j}); \hat a))= \lambda_{\hat a}\prod_{0\leq i\leq f-1, 0 \leq j\leq e-1}\psi_{i, j}^{r_{i, j}},$
    where $\psi_{i, j}$ is a certain crystalline character such that $\HT_{i', j'}(\psi_{i, j})=\{0\}$ for $(i',j')\neq (i, j)$ and $\HT_{i, j}(\psi_{i, j})=\{1\}$, and $\lambda_{\hat a}$ is the unramified character of $G_K$ which sends the arithmetic Frobenius to $\hat a$.
 \end{enumerate}
 \end{lemma}
\begin{proof}
(1) is easy. For (2), we imitate the proof of \cite[Lem.~6.3]{GLS14}. It suffices to show existence of $\hat \huaM((r_{i, j}); \hat a)$. Similar as in \textit{loc. cit.}, consider $\huaM((\mathbbm{1}_{i, j}); 1)$ where $(\mathbbm{1}_{i, j})$ is the matrix where the only nonzero element is in the $(i, j)$-position, which is $1$. This is a height $1$ Kisin module, and similarly as in the argument of \textit{loc. cit.} (by applying our Theorem \ref{shapecorollary}), $T_{\huaS}(\huaM((\mathbbm{1}_{i, j}); 1))$ extends to a crystalline character $\psi_{i, j}$ such that $\HT_{i', j'}(\psi_{i, j})=\{0\}$ for $(i',j')\neq (i, j)$ and $\HT_{i, j}(\psi_{i, j})=\{1\}$. Then one can continue the argument as in \cite[Lem.~6.3]{GLS14} to conclude.
\end{proof}

\section{Shape of upper triangular Kisin modules with $k_E$-coefficients} \label{section: shape kisin}
In this section, when $K/\Qp$ is ramified of degree $e=2$ (in particular, $K/\Qp$ is tamely ramified since $p>2$), we study the shape of upper triangular Kisin modules (with $k_E$-coefficients) coming from reductions of crystalline representations. We divide this section into two steps. In the first step, we determine the information on the diagonal of the matrix of $\varphi$ for $\barm$. In the second step, we determine the structure of the full matrix of $\varphi$.

\subsection{Shape of the diagonal} We first list some very elementary linear algebra lemmas. By writing out these lemmas first, it will make the proof of our main result Proposition \ref{tameinertia} more transparent (in particular, we would not need to introduce too many notations there).

\begin{lemma} \label{lemma: property Z}
Let $A \in \GL_d(k_E + u^{\Delta}k_E\llb u\rrb)$ for some $\Delta \geq 0$, where $\GL_d(k_E + u^{\Delta}k_E\llb u\rrb)$ denotes the set of invertible matrices with all elements in the ring $k_E + u^{\Delta}k_E\llb u\rrb$. Then there exists a unique unipotent (i.e., upper triangular with all diagonal elements being $1$) matrix $C \in \GL_d(k_E)$ such that $AC=(m_{x, y})$ satisfies the following: there exists an ordering $\{k_1, \ldots, k_d\}$ of $\{1, \ldots, d\}$, such that for each $1 \leq x \leq d$,
\begin{itemize}
\item $m_{k_x, x}$ is the top most element in $\col_x(AC)$ that is a unit, and
\item the elements above $m_{k_x, x}$ satisfy $u^{\Delta} \mid m_{z, x} , \forall z<k_x$ (note that when $\Delta=0$, then $m_{z, x}=0 , \forall z<k_x$),
\item the elements to the right of $m_{k_x, x}$ satisfy $m_{k_x, y}=0, \forall y>x$.
\end{itemize}
\end{lemma}

\begin{proof}
This is extracted from the beginning of the proof of \cite[Lem. 2.4]{Gao15}, where $\Delta=p$.
\end{proof}

When the conclusion of Lemma \ref{lemma: property Z} is satisfied, we say that the matrix $AC$ satisfies (Property Z).

\begin{lemma} \label{matrix}
Let $M_1=M_2M_3M_4$, where
\begin{itemize}
  \item $M_1 \in \Mat_d(k_E\llb u\rrb)$ upper triangular, such that $\diag M_1=[c_1u^{t_1}, \ldots, c_du^{t_d}]$ where $t_i \geq 0, c_i \in k_E\llb u \rrb^{\times}$.
  \item $M_2 \in \Mat_d(k_E\llb u\rrb)$.
  \item $M_3=[u^{r_1}, \ldots, u^{r_d}]$, where $0\leq r_1 \leq \ldots \leq r_d \leq \Delta$ for some $\Delta \geq 0$.
  \item $M_4 \in \GL_d(k_E+u^{\Delta}k_E\llb u\rrb)$, i.e., $M_4=M_5+ u^{\Delta}M_6$ for some $M_5 \in \GL_d(k_E), M_6 \in \Mat_d(k_E\llb u\rrb)$.
\end{itemize}
Then we have:
\begin{enumerate}
  \item There exists an ordering $\{k_1, \ldots, k_d\}$ of $\{1, \ldots, d\}$ such that $r_{k_x} \leq t_x, \forall x$.
  \item  If $M_2$ is furthermore invertible, then $r_{k_x}=t_x, \forall x$.
\end{enumerate}
\end{lemma}
\begin{proof}
This is easy generalization of \cite[Lem. 2.4]{Gao15} (where $\Delta=p$). We sketch the proof.
For $M_4$, by Lemma \ref{lemma: property Z}, we can find a unipotent $M_7 \in \GL_d(k_E)$ such that $M_4M_7$ satisfies (Property Z). It is easy to see that $u^{r_{k_x}} \mid \col_x(M_3M_4M_7), \forall x$, where $\{k_1, \ldots, k_d\}$ is the ordering of $\{1, \ldots, d\}$ in the conclusion of Lemma \ref{lemma: property Z}. So $u^{r_{k_x}} \mid \col_x(M_1M_7)$, and so $u^{r_{k_x}} \mid u^{t_x}$ since $M_7$ is unipotent, i.e., $r_{k_x} \leq t_x, \forall x$.

Note that for our Statement (2), we need $M_2$ to be invertible in order to apply the determinant argument at the end of \cite[Lem. 2.4]{Gao15}.
\end{proof}

\begin{lemma} \label{lemma: Q decomp}
With exactly the same notations in Lemma \ref{matrix} (including in its proof). Suppose furthermore that there exists $\delta >0$ such that $r_{x+1}-r_x \geq \delta$ for all $1 \leq x\leq d-1$ and $\Delta-r_d \geq \delta$. Then $M_3M_4M_7=Q[u^{r_{k_1}}, \ldots, u^{r_{k_d}}]$ for some $Q \in \GL_d(k_E+u^{\delta}k_E\llb u \rrb )$.
\end{lemma}
\begin{proof}
By the proof in Lemma \ref{matrix}, $u^{r_{k_x}} \mid \col_x(M_3M_4M_7), \forall x$, which shows the existence of $Q$. It suffices to show that $Q \in \GL_d(k_E+u^{\delta}k_E\llb u \rrb )$. Write $Q=(q_{x, y})$, and we now only prove that the elements in $\col_1(Q)$ are in the ring $k_E+u^{\delta}k_E\llb u \rrb$ (for the other columns, just use similar argument). Note that $M_4M_7$ satisfies (Property Z), if we write $M_3M_4M_7 =(g_{i, j})$, then in $\col_1(M_3M_4M_7)$, we have
\begin{itemize}
  \item $ u^{\Delta} \mid g_{k, 1}, \forall k<k_1$, and so $u^{r_{k_1} +\delta } \mid  g_{k, 1}, \forall k<k_1$;
  \item $u^{r_{k_1}} \parallel g_{k_1, 1}$;
  \item  $u^{r_{k}} \mid g_{k, 1}, \forall k>k_1$,  and so $u^{r_{k_1} +\delta} \mid g_{k, 1}, \forall k>k_1$.
\end{itemize}
Since $g_{x, 1} = q_{x, 1}u^{r_{k_1}}$, so in $\col_1(Q)$, we have $u^{\delta} \mid q_{k, 1} \forall k \neq k_1$ and $q_{k_1, 1} \in k_E^{\times}$, and we are done.
\end{proof}

\begin{prop}\label{tameinertia}
With notations in (CRYS), suppose $e=2$ (so we always have $p \nmid e$) and $\rhobar$ is upper triangular. Then $\barm$ is upper triangular, i.e., $\barm \in \mathcal E(\barn_d, \ldots, \barn_1)$, where $\barn_x=\barm(t_{0, x}, \ldots, t_{f-1, x};a_x)$ are some rank-1 modules.
If the following assumptions are satisfied:
\begin{itemize}
  \item   $r_{i, 0, x+1}-r_{i, 0, x} \geq r_{i, 1, d}, \forall x, i$, and
  \item  $r_{i, 0, d}+r_{i, 1, d} \leq p, \forall i$.
\end{itemize}
Then we have $t_{i, x} =r_{i,0, \sigma_{i, 0}(x)}+ r_{i,1, \sigma_{i, 1}(x)} \forall i, x$, where $\sigma_{i, 0}, \sigma_{i, 1}$ are orderings of $\{1, \ldots, d\}$.
\end{prop}

\begin{proof}
Via Theorem \ref{shapecorollary}, we have
$\varphi(\bolde_{i-1}) = \bolde_{i} X_i  \Lambda_{i,1} Z_{i,1} \Lambda_{i,0}$ for some basis $\bolde_i$ of $\barm_i$.
Here, by $X_i,  \Lambda_{i,1}, Z_{i, 1}, \Lambda_{i,0},$ we really mean their reductions modulo $\omega_E$.
Similarly as in \cite[Prop. 2.3]{Gao15}, $\barm$ is upper triangular, and there exists another basis $\boldf_i$ of $\barm_i$ such that
$\varphi(\boldf_{i-1}) = \boldf_{i} F_i$ where $F_i$ is upper triangular with $\diag F_i=[(a_x)_i u^{t_{i, x}}]_{x=1}^d $. Suppose $\bolde_i=\boldf_i T_i$, then we have
$ F_i =T_i X_i  \Lambda_{i,1} Z_{i,1} \Lambda_{i,0} \varphi(T_{i-1}^{-1}). $
Now, let us drop $i$ from all the subscripts, so
$$F=TX\Lambda_1Z_1\Lambda_0S,$$
where $S=\varphi(T_{i-1}^{-1}) \in \GL_d(k_E\llb u^p \rrb) \subset \GL_d(k_E+u^pk_E\llb u\rrb)$.
Now we can apply Lemma \ref{matrix}(1), where we let $M_1=F, M_2=TX\Lambda_1Z_1, M_3=\Lambda_0, M_4=S$ with $\Delta=p$, then we have $r_{0, k_x} \leq t_x$, so we can write $t_x=r_{0, k_x} +\gamma_x$ for some $\gamma_x \geq 0, \forall x$.

Also by the proof of Lemma \ref{matrix}(1), there exists a unipotent matrix $M_7$ such that $FM_7=B[u^{r_{0, k_x}}]_{x=1}^d$ and $\Lambda_0SM_7=Q[u^{r_{0, k_x}}]_{x=1}^d$ for some $B$ and $Q$. We must have $B \in \Mat_d(k_E\llb u\rrb)$ is upper triangular (since $M_7$ is upper triangular), and $Q\in \GL_d(k_E+ u^{r_{i, 1, d}}k_E\llb u\rrb)$ (by applying Lemma \ref{lemma: Q decomp} with $\delta=r_{i, 1, d}$).

So we have $FM_7=B[u^{r_{0, k_x}}]_{x=1}^d=TX\Lambda_1Z_1Q[u^{r_{0, k_x}}]_{x=1}^d$, and so $B=TX\Lambda_1Z_1Q$. Now we can apply Lemma \ref{matrix}(2), where we let $M_1=B, M_2=TX, M_3=\Lambda_1, M_4=Z_1Q$ with $\Delta=r_{i, 1, d}$ (note that $Z_1 \in \GL_d(k_E)$ by Theorem \ref{shapecorollary}).
so we have that
$u^{r_{1, k_x'}} \parallel b_{x, x} $  on the diagonal of $B$, where $\{k_1', \ldots, k_d'\}$ is an ordering of $\{1, \ldots, d\}$.
Since $M_7$ is unipotent, we finally have that $u^{r_{0, k_x} + r_{1, k_x'}  } \parallel f_{x, x}$ on the diagonal of $F$. That is, $t_{x}=r_{0, k_x} + r_{1, k_x'} $.
\end{proof}

\subsection{Shape of upper triangular Kisin modules}
\renewcommand{\labelitemii}{$\diamond$}

Let $X$ be an upper triangular matrix in $\Mat_d(k_E[u])$ of the shape
  $\left(
 \begin{array}{ccccc}
 u^{s_1} & & x_{i,j}\\
         & \ddots & \\
         & & u^{s_d}
 \end{array}
\right),$
where $0 \leq s_i \leq p$ are \emph{distinct} integers.

\begin{itemize}
\item We call $X$ satisfies the property (DEG) if $\deg(x_{i, j}) < s_j, \forall i<j$.

\item We call $X$ satisfies property (P) if $x_{i, j}=u^{s_i}y_{i, j}, \forall i<j$, where
\begin{itemize}
  \item $y_{i, j}=0$ if $s_i > s_j$, and
  \item  $y_{i, j} \in k_E$ if $s_i<s_j$.
\end{itemize}
\end{itemize}

Let $X$ be as above which satisfies (DEG), recall that in the proof of \cite[Lem. 4.3]{Gao15}, we call the following procedure an \textbf{allowable procedure} for $X$:
$$X \rightsquigarrow X'=X(Id - \Mat_d(c_{i, j})),$$
where $1\leq i<j \leq d$ are two numbers such that $s_i <s_j$ and $c_{i, j} \in k_E$; and $\Mat_d(c_{i, j})$ is the matrix where the only nonzero element is at $(i, j)$-position, and the element is precisely $c_{i, j}$.
This allowable procedure has the following properties:
\begin{itemize}
  \item $X'$ still satisfies (DEG).
  \item $X$ satisfies (P) if and only if $X'$ satisfies (P).
  \item When $X$ satisfies (P), one can apply finite times of \emph{allowable procedures} to change $X$ to the diagonal matrix $[u^{s_1}, \ldots, u^{s_d}]$.
\end{itemize}

\begin{lemma}\label{shapelemma}
Let $t_1, \ldots, t_d$ be distinct integers, and let $\delta_1, \ldots, \delta_d$ be nonnegative integers, such that if we let $\delta=\max\{\delta_i\}$, then $|t_{i_1}-t_{i_2}| > \delta, \forall i_1 \neq i_2$.
Let $X$ be an upper triangular matrix in $\Mat_d(k_E[u])$ of the shape
  $$\left(
 \begin{array}{ccccc}
 u^{t_1+\delta_1} & & x_{i,j}\\
         & \ddots & \\
         & & u^{t_d+\delta_d}
 \end{array}
\right),$$
where $\diag X =[u^{t_i+\delta_i}]_{i=1}^d$. Let $A \in \GL_d(k_E)$.
If we have
\begin{itemize}
  \item $X$ satisfies (DEG).
  \item $u^{t_i} \mid \col_i(XA), \forall i.$
\end{itemize}
Then $X=X_1 X_0$, where
\begin{itemize}
  \item $X_0$ is upper triangular, $\diag X_0 =[u^{t_i}]_{i=1}^d$, and satisfies property (P).
  \item $X_1$ is upper triangular, $\diag X_1 =[u^{\delta_i}]_{i=1}^d$, and satisfies (DEG).
  \item $X_0A=B[u^{t_i}]_{i=1}^d$ for some $B \in \GL_d(k_E+ u^{\delta}k_E\llb u\rrb)$.
\end{itemize}
\end{lemma}

\begin{proof}
It is clear that \cite[Lem. 4.3]{Gao15} is the special case of our lemma for $\delta=0$, and the proof of our lemma is very similar to \textit{loc. cit.}. Let us first remark here that it is easy to see if we apply an allowable procedure to $X$, namely, change $X$ to some $X(Id - \Mat_d(c_{i, j}))$, and change $A$ to $(Id - \Mat_d(c_{i, j}))^{-1}A$. Then the conclusion of the lemma still holds.

We prove the lemma by induction. The case $d=1$ is trivial. Suppose it is true for dimension less than $d$, and now let the dimension become $d$. Similarly as in \textit{loc. cit.}, we prove two special cases first. Note here that $t_1+\delta_1, \ldots, t_d+\delta_d$ are distinct integers, and whenever we have $t_i>t_j$, we must have $t_i+\delta_i > t_j +\delta_j$, and in fact $t_i >t_j +\delta_j$.

\textbf{(Case 1)}. Suppose $t_d$ is maximal in $\{t_1, \ldots, t_d\}$. Note that we must have $t_d+\delta_d=\max \{t_1+\delta_1, \ldots, t_d+\delta_d\}$.

Because $u^{t_d} \mid X  \left(
 \begin{array}{ccccc}
 a_{1, d}  \\
    \vdots     \\
  a_{d, d}
 \end{array}
\right), $
so we have $u^{t_d} \mid u^{t_{d-1}+\delta_{d-1}}a_{d-1, d} + x_{d-1, d}a_{d, d}$. Note here that $\deg(x_{d-1, d})<t_d+\delta_d$ (\emph{not} $<t_d$ as in \cite[Sublem. 4.5]{Gao15}), so we can \emph{not} conclude that $u^{t_{d-1}+\delta_{d-1}}a_{d-1, d} + x_{d-1, d}a_{d, d}=0$ as in \textit{loc. cit.}. However, we still have $a_{d, d}\neq 0$. Since otherwise, we have $a_{d-1, d}=0$, and similarly as in \textit{loc. cit.}, we can then show that $a_{i, d}=0, \forall i$, which is impossible.
So we must have
$$x_{d-1, d} =u^{t_{d-1}+\delta_{d-1}}y_{d-1, d} +u^{t_d}x'_{d-1, d}$$
for some $y_{d-1, d} \in k_E, x'_{d-1, d} \in k_E[u]$ with $\deg(x'_{d-1, d})<\delta_d$.
Then similarly as in \textit{loc. cit.}, we can apply the following allowable procedure:
$$ X \rightsquigarrow X (Id - \Mat_d(y_{d-1, d})), \quad A \rightsquigarrow (Id - \Mat_d(y_{d-1, d}))^{-1}A,$$
Then $a_{d-1, d}$ becomes $0$, and $x_{d-1, d}$ becomes $u^{t_d}x'_{d-1, d}$.

Repeat the above argument and operations, in the end we will have
$$ X A = \left(
 \begin{array}{ccccc}
 X_{d, d} & u^{t_d}(x'_{i, d})_{i=1}^{d-1}  \\
 0 & u^{t_d+\delta_d}
 \end{array}
\right)
\left(
 \begin{array}{ccccc}
 A_{d, d} & 0  \\
 (a_{d, j})_{j=1}^{d-1} & a_{d, d}
 \end{array}
\right)
=
\left(
 \begin{array}{ccccc}
 X_{d, d}A_{d, d}+u^{t_d}P & u^{t_d}a_{d, d}(x'_{i, d})_{i=1}^{d-1}  \\
 u^{t_d+\delta_d}(a_{d, j})_{j=1}^{d-1} & u^{t_d+\delta_d}a_{d, d}
 \end{array}
\right),
$$
where $P$ is the $(d-1)\times (d-1)$ matrix formed by the product of the $(d-1)\times 1$-matrix $(x'_{i, d})_{i=1}^{d-1}$ and the $1\times (d-1)$-matrix $(a_{d, j})_{j=1}^{d-1}$.
Then we clearly have $u^{t_i} \mid \col_i(X_{d, d}A_{d, d})$ (since $t_d$ is maximal), so by induction hypothesis, we can decompose $X_{d, d}=(X_{d, d})_1(X_{d, d})_0$, and $(X_{d, d})_0 A_{d, d}=B_{d, d}[u^{t_i}]_{i=1}^{d-1}$.
So we can decompose
$$ X =X_1X_0  = \left(
 \begin{array}{ccccc}
 (X_{d, d})_1 &  (x'_{i, d})_{i=1}^{d-1}  \\
 0 & u^{ \delta_d}
 \end{array}
\right)
\left(
 \begin{array}{ccccc}
 (X_{d, d})_0 & 0  \\
 0 & u^{ t_d}
 \end{array}
\right),
$$
and so $X_1$ satisfies (DEG).
We also have that
$$ X_0A   = \left(
 \begin{array}{ccccc}
 (X_{d, d})_0 &  0  \\
 0 & u^{ t_d}
 \end{array}
\right)
\left(
 \begin{array}{ccccc}
 A_{d, d}  &  0  \\
 (a_{d, j})_{j=1}^{d-1} & a_{d, d}
 \end{array}
\right)
=\left(
 \begin{array}{ccccc}
 B_{d, d}  &  0  \\
 (u^{t_d-t_j}a_{d, j})_{j=1}^{d-1} & a_{d, d}
 \end{array}
\right)
\left(
 \begin{array}{ccccc}
 [u^{t_i}]_{i=1}^{d-1} &  0  \\
 0 & u^{t_d}
 \end{array}
\right),
$$
so $B \in \GL_d(k_E+ u^{\delta}k_E\llb u\rrb)$, and we are done for the proof of Case 1.

\textbf{(Case 2)}. Suppose $t_1$ is maximal in $\{t_1, \ldots, t_d\}$. Similarly as in \cite[Sublem. 4.6]{Gao15}, the situation is:
$$ X A = \left(
 \begin{array}{ccccc}
 u^{t_1+\delta_1} & (x_{1, j})_{j=2}^d\\
 0 & X_{1, 1}
 \end{array}
\right)
\left(
 \begin{array}{ccccc}
 a_{1, 1} & (a_{1, j})_{j=2}^d  \\
 0 & A_{1, 1}
 \end{array}
\right)
=
\left(
 \begin{array}{ccccc}
 u^{t_1+\delta_1}a_{1, 1} & u^{t_1+\delta_1}(a_{1, j})+(x_{1, j})A_{1, 1}  \\
 0 & X_{1, 1}A_{1, 1}
 \end{array}
\right).
$$

Then we can apply induction hypothesis on $X_{1, 1}$, and write $X_{1, 1}= (X_{1, 1})_1 (X_{1, 1})_0$ as in the conclusion of our lemma. Since $(X_{1, 1})_0$ satisfies (P), we can apply finite times of allowable procedures to change $(X_{1, 1})_0$ to the diagonal matrix $[u^{t_2}, \ldots, u^{t_d}]$. So we have
$$XA=   \left(
 \begin{array}{ccccc}
 u^{t_1+\delta_1} & (x_{1, j})_{j=2}^d\\
 0 & (X_{1, 1})_1[u^{t_2}, \ldots, u^{t_d}]
 \end{array}
\right)
\left(
 \begin{array}{ccccc}
 a_{1, 1} & (a_{1, j})_{j=2}^d  \\
 0 & A_{1, 1}
 \end{array}
\right).
$$
We claim that we have $u^{t_j} \mid x_{1, j}, \forall 2\leq j \leq d$. To prove the claim, suppose $t_{k_1} = \max \{t_2, \ldots, t_d\}$, so we have
$$u^{t_{k_1}} \mid X\left(
 \begin{array}{ccccc}
 a_{1, k_1}\\
 \vdots\\
 a_{d, k_1}
 \end{array}
\right), \textnormal{and so }
u^{t_{k_1}} \mid (X_{1, 1})_1[u^{t_2}, \ldots, u^{t_d}]
\left(
 \begin{array}{ccccc}
 a_{2, k_1}\\
 \vdots\\
 a_{d, k_1}
 \end{array}
\right).
$$
Using that $(X_{1, 1})_1$ is upper triangular and $t_{k_1}> t_i+\delta_i, \forall i \neq 1, k_1$, we can easily find that $a_{i, k_1} =0, \forall i \neq 1, k_1$. Now we must have $a_{k_1, k_1} \neq 0$ because of the shape of $\col_1(A)$ (the only nonzero element of $\col_1(A)$ is $a_{1, 1}$, so if $a_{k_1, k_1}=0$, then $A$ will not be invertible). Consider the first row of $X \col_{k_1}(A)$, we will have $u^{t_{k_1}} \mid u^{t_1+\delta_1}a_{1, k_1} + x_{1, k_1}a_{k_1, k_1}$, and so $u^{t_{k_1}} \mid x_{1, k_1}$. Now let $t_{k_2} =\textnormal{max} \{\{t_1, \ldots, t_d\}-\{t_1, t_{k_1}\}\}$, and repeat the above argument just similarly as we did in \cite[Sublem. 4.6]{Gao15}. In the end, we will have that $u^{t_j} \mid x_{1, j}, \forall 2 \leq j \leq d$. We also note that for the matrix $A$, we must have $a_{i, j}=0$ unless $t_i \geq t_j$.

So we have
$$X =X_1X_0= \left(\begin{array}{ccccc}
u^{\delta_1} & (\frac{x_{1, j}}{u^{t_j}})_{j=2}^d\\
0 & (X_{1, 1})_1
\end{array} \right)
 \left(\begin{array}{ccccc}
 u^{t_1} & 0\\
 0 & [u^{t_j}]_{j=2}^d
 \end{array} \right),
$$
so $X_1$ satisfies (DEG).
We also have
$$X_0A=
\left(\begin{array}{ccccc}
 u^{t_1} & 0\\
 0 & [u^{t_j}]_{j=2}^d
 \end{array} \right)
\left(\begin{array}{ccccc}
 a_{1, 1} & (a_{1, j})_{j=2}^d\\
 0 & A_{1, 1}
 \end{array} \right)
=
\left(\begin{array}{ccccc}
 a_{1, 1} & (a_{1, j} \frac{u^{t_1}}{u^{t_j}})_{j=2}^d\\
 0 & B_{1, 1}
 \end{array} \right)
\left(\begin{array}{ccccc}
 u^{t_1} & 0\\
 0 & [u^{t_j}]_{j=2}^d
 \end{array} \right),
$$
and so $B \in \GL_d(k_E+ u^{\delta}k_E\llb u\rrb)$.

\textbf{(General case)}. Argue similarly as in \cite[Lem. 4.3]{Gao15}, using the results we obtained in (Case 1) and (Case 2).
\end{proof}

The following lemma is easy corollary of the above lemma. Writing it out will greatly simplify the proof of our main Proposition \ref{shape}.
\begin{lemma} \label{lemma: shape lemma with extra power}
Lemma \ref{shapelemma} still holds if we replace the sentence ``Let $A \in \GL_d(k_E)$" to the sentence ``Let $A \in \GL_d(k_E + u^{\gamma}k_E\llb u\rrb )$ where $\gamma \geq \max\{t_i\}+\max \{\delta_i\}$".
\end{lemma}
\begin{proof}
Let $A=A_1 + u^{\gamma}A_2$ where $A_1 \in \GL_d(k_E), A_2 \in \Mat_d(k_E\llb u\rrb)$. Then apparently we have $u^{t_i} \mid \col_i(XA_1)$. So we can apply Lemma \ref{shapelemma} (to the pair $X, A_1$) to conclude $X=X_1X_0$ where $X_0$ satisfy (P) and $X_1$ satisfy (DEG). We also have $X_0A_1 =B_1[u^{t_i}]_{i=1}^d$ for $B_1 \in \GL_d(k_E +u^{\delta}k_E\llb u \rrb)$ by \textit{loc. cit.}. So $X_0A = X_0(A_1 + u^{\gamma}A_2)= (B_1+ X_0A_2[\frac{u^{\gamma}}{u^{t_i}}]_{i=1}^d) [u^{t_i}]_{i=1}^d$. And we clearly have $B_1+ X_0A_2[\frac{u^{\gamma}}{u^{t_i}}]_{i=1}^d \in  \GL_d(k_E +u^{\delta}k_E\llb u \rrb)$.
\end{proof}

\begin{prop} \label{shape}
With notations in (CRYS), suppose $e=2$ and $\rhobar$ is upper triangular.
Suppose furthermore
\begin{itemize}
  \item (A1): $r_{i, 0, x+1}-r_{i, 0, x} > r_{i, 1, d}, \forall x, i$, and
  \item (A2):  $r_{i, 0, d}+r_{i, 1, d} \leq p-2, \forall i$;
\end{itemize}
Then there exists basis $\boldf_i$ of $\barm_i$, such that if we write $\varphi(\boldf_{i-1})=\boldf_i F_i$, then
\begin{itemize}
  \item $F_i$ is upper triangular with $\diag F_i = [(a_x)_i u^{ r_{i, 0, \sigma_{i, 0}(x)} + r_{i, 1, \sigma_{i, 1}(x)}}]_{x=1}^d$ where $\sigma_{i,0}$ and $\sigma_{i,1}$ are orderings of $\{1, \ldots, d\}$; $F_i$ satisfies (DEG); and

  \item $F_i=C_i(F_i)_1(F_i)_0$ where $C_i = [(a_x)_i]_{x=1}^d$ for some $a_x \in k_E^{\times}$; both $(F_i)_1, (F_i)_0$ are upper triangular with $\diag((F_i)_0)=\widetilde \Lambda_{i, 0} =[u^{r_{i, 0, \sigma_{i, 0}(x)}}]_{x=1}^d$ and $\diag((F_i)_1)=\widetilde \Lambda_{i, 1} =[u^{r_{i, 1, \sigma_{i, 1}(x)}}]_{x=1}^d$; furthermore,
  \item Both $(F_i)_1, (F_i)_0$ satisfies property (P).
\end{itemize}
\end{prop}

\begin{proof}
The existence of $\boldf_i$ such that $F_i$ satisfies (DEG) is a consequence of \cite[Prop. 2.2]{Gao15}, Proposition \ref{tameinertia}, as well as the assumption (A2). Note that we need assumption (A2) to avoid the situation in Statement (3) of \cite[Prop. 2.2]{Gao15} (see also \cite[\S 3]{Gao15} for some explanation).

Note that by Proposition \ref{tameinertia}, $\diag F_i=[(a_x)_i u^{r_{i, 0, \sigma_0(x) + r_{i,1,\sigma_1(x)}}} ]_{x=1}^d$.   We now prove that $F_i$ satisfies the other properties.
We drop $i$ from the subscripts (except we keep the subscript on $X_i$, to avoid confusion with $X$ in Lemma \ref{shapelemma} and Lemma \ref{lemma: shape lemma with extra power}), so we write similarly as in Proposition \ref{tameinertia}:
$$F = TX_i\Lambda_1 Z_1 \Lambda_0 S.$$
Let $W_0 \in \GL_d(k_E)$ such that $W_0 \Lambda_0 W_0^{-1} = \widetilde{\Lambda}_0= [u^{r_{i, 0, \sigma_0(x)}}]_{x=1}^d$. Then
$$ FC_i^{-1}(C_iS^{-1}W_0^{-1}) = TX_i \Lambda_1 Z_1 W_0^{-1}\widetilde{\Lambda}_0, \textnormal{ recall here } C_i= [(a_x)_i]_{x=1}^d.$$
So we can apply Lemma \ref{lemma: shape lemma with extra power}, where we let $X=FC_i^{-1}, A=C_iS^{-1}W_0^{-1}$ with $\gamma=p$. So we will have $FC_i^{-1}=F_1F_0$ with $F_0$ satisfying property (P). Now it suffices to show that $F_1$ also satisfies property (P).

Note that in our situation, $\delta$ in Lemma \ref{lemma: shape lemma with extra power} is precisely our $r_{i, 1, d}$. So by the conclusion of \textit{loc. cit.}, there exists $B \in \GL_d(k_E +u^{r_{i, 1, d}}k_E\llb u \rrb)$ such that $F_0 C_iS^{-1}W_0^{-1} = B\widetilde{\Lambda}_0$. So we will have $F_1B =TX_i\Lambda_1Z_1W_0^{-1}$, that is $ F_1 B W_0 Z_1^{-1} =TX_i\Lambda_1.$
Now let $W_1 \in \GL_d(k_E)$ such that $W_1 \Lambda_1 W_1^{-1} = \widetilde{\Lambda}_1= [u^{r_{i, 1, \sigma_1(x)}}]_{x=1}^d=\diag F_1$, so we have
$$ F_1 BW_0Z_1^{-1}W_1^{-1} =TX_iW_1^{-1}\widetilde{\Lambda}_1 .$$
Now we can apply Lemma \ref{lemma: shape lemma with extra power} again, where we let $X=F_1, A=B W_0 Z_1^{-1}W_1^{-1}$ with $\gamma=r_{i, 1, d}$ and $\delta=0$ (in this case, as we mentioned in the beginning of the proof of Lemma \ref{shapelemma}, it is indeed \cite[Lem. 4.3]{Gao15}, and we have $X=X_0$ satisfying (P)). So now we conclude that $F_1$ satisfies (P), and we are done.
\end{proof}

\begin{rem} \label{rem:higher e}
We have analysed the shape of $\varphi_{\barm}$ in this section, for $e=2$. It is clear we can in fact generalize this section for \emph{any} tame ramification index $e$, with some additional assumptions. For notational simplicity, let us just discuss about $e=3$ in this remark, the higher $e$ case is similar.

When $e=3$, in order for the arguments of Proposition \ref{tameinertia} and Proposition \ref{shape} to work through, we will need the following assumptions:
\begin{itemize}
  \item (A1)-0 : $r_{i, 0, x+1}-r_{i, 0, x} > r_{i, 1, d} +r_{i, 2, d}, \forall x, i$,
  \item (A1)-1 : $r_{i, 1, x+1}-r_{i, 1, x} > r_{i, 2, d}, \forall x, i$,
  \item (A2): $r_{i, 0, d}+r_{i, 1, d}  +r_{i, 2, d}\leq p-2, \forall i$.
\end{itemize}
Note that Condition (A1)-0 says that the weights in $\HT_{i, 0}$ are enough separated with respect to both $\HT_{i, 1}$ and $\HT_{i, 2}$; and Condition (A1)-1 says that weights in $\HT_{i, 1}$ are enough separated with respect to $\HT_{i, 2}$.

However, these assumptions will never be satisfied for the crystalline representations that will be useful for our global application. Namely, the Serre type crystalline representations in Definition \ref{definition serre type} will never satisfy the (A1)-1 above, so we choose not to write out these generalizations.
\end{rem}

\section{Crystalline liftings and Serre weight conjectures} \label{section: final section lifting}

In this section, we prove our local theorem on crystalline liftings, as well as its application to weight part of Serre's conjectures. The proof of our crystalline lifting theorem follows the same strategy of \cite[Thm. 7.4]{Gao15}. We will freely use notations and results in \cite[\S 5, \S 6, \S 7]{Gao15}. In particular, we will use the various $\Ext(*, *)$'s defined in various categories, and their properties.
In this section, we always assume the ramification index $e=e(K/\Qp)=2$.

\subsection{Crystalline lifting theorem}

Recall that if $\huaL, \huaL' \in \Mod_{\huaS_{\mathcal O_E}}^{\varphi}$, then we can let $\Ext(\huaL', \huaL)$ be the set of short exact sequences $0 \to \huaL \to \huaN \to \huaL' \to 0$ in the category $\Mod_{\huaS_{\mathcal O_E}}^{\varphi}$, modulo the natural equivalence relation as in \cite[Def. 5.2]{Gao15}. We can similarly define the set of $\Ext(*, *)$'s in the categories $\Mod_{\huaS_{\mathcal O_E}}^{\varphi, \Ghat}$, $\Mod_{\huaS_{k_E}}^{\varphi}$, and $\Mod_{\huaS_{k_E}}^{\varphi, \Ghat}$.
These $\Ext(*, *)$'s have natural $\mathcal O_E$-module (or $k_E$-vector space) structures by \cite[Prop. 5.4]{Gao15}.

Now we
define certain set of upper triangular extensions of rank-1 modules in $\Mod_{\huaS_{k_E}}^{\varphi}$, similar to the spirit of \cite[Def. 5.8]{Gao15}.

\begin{defn} \label{defn: phi shape extensions}
Let $\barn_1, \ldots, \barn_d$ (resp. $\barhatn_1, \ldots, \barhatn_d$) be rank-1 modules in $\Mod_{\huaS_{k_E}}^{\varphi}$ (resp. $\Mod_{\huaS_{k_E}}^{\varphi, \hat G}$).
\begin{enumerate}
  \item Let $\mathcal E_{\vshape}(\barn_d, \ldots, \barn_1) \subset \mathcal E(\barn_d, \ldots, \barn_1)$ be the subset consisting of elements $\barm$ such that there exists a basis $\boldf_i$ of $\barm_i$, $\varphi(\boldf_{i-1})=\boldf_i F_i$, and $F_i$ is of the shape in Proposition \ref{shape} for each $i$. That is $F_i = [(a_x)_i]_{x=1}^d(F_i)_1(F_i)_0$, where both $(F_i)_1, (F_i)_0$ satisfy property (P).

  \item Let $\mathcal E_{\vtshape}(\barhatn_d, \ldots, \barhatn_1) \subset \mathcal E(\barhatn_d, \ldots, \barhatn_1)$
 be the subset consisting of elements $\barhatm$ such that there exists a basis $\boldf_i$ of $\barm_i$ such that
   $\varphi(\boldf_{i-1})=\boldf_i F_i$ with $F_i$ of the shape in Proposition \ref{shape}, and
   $\tau(1\otimes_{\varphi}\boldf_i)=(1\otimes_{\varphi}\boldf_i) Z_i$ with $Z_i$ of the shape in \cite[Lem. 5.7]{Gao15}, namely,
        \begin{itemize}
         \item $Z_i=(z_{i, x, y})  \subset \Mat(R\otimes_{\Fp}k_E)$ is upper triangular.
  \item On the diagonal, $v_R(z_{i, x, x}-1) \geq \frac{p^2}{p-1}, \forall x$.
  \item On the upper right corner, $v_R(z_{i, x, y}) \geq \frac{p^2}{p-1}, \forall x<y$.
        \end{itemize}

\item Suppose $\barm \in \mathcal E_{\varphi-\rm shape}(\barn_d, \ldots, \barn_1), \overline{\huaM'} \in \mathcal E_{\varphi-\rm shape}(\barn_{d'}', \ldots, \barn_{1}').$
 We can define the set $\Ext_{\vshape}(\barm, \barm')$ analogously as in \cite[Def. 5.8]{Gao15}. And similarly for $\Ext_{\vtshape}(\barhatm, \barhatm')$.
\end{enumerate}
\end{defn}

\begin{prop} \label{shapesubmodule}
With notations in Definition \ref{defn: phi shape extensions}, we have the following.
  \begin{enumerate}
   \item $\Ext_{\vshape}(\barm, \barm')$ is a sub-vector space of $\Ext(\barm, \barm')$.
   \item  $\Ext_{\vtshape}(\barhatm, \barhatm')$ is a sub-vector space of $\Ext(\barhatm, \barhatm').$
  \end{enumerate}
\end{prop}
\begin{proof}
Similar to \cite[Prop. 5.9]{Gao15}, by using the following fact of matrix multiplications:
If
$\squaremat{A}{C^{(i)}}{0}{A'}=
\squaremat{A_1}{C^{(i)}_1}{0}{A'_1}
\squaremat{A_0}{C^{(i)}_0}{0}{A'_0},
$ for $i=1, 2$, then
$\squaremat{A}{aC^{(1)}+bC^{(2)}}{0}{A'}=
\squaremat{A_1}{aC^{(1)}_1+bC^{(2)}_1}{0}{A'_1}
\squaremat{A_0}{aC^{(1)}_0+bC^{(2)}_0}{0}{A'_0}.
$
\end{proof}

\begin{thm} \label{mainlocal}
With notations in (CRYS), suppose $e=2$ and $\rhobar$ is upper triangular.
Suppose furthermore
\begin{itemize}
  \item (A1): $r_{i, 0, x+1}-r_{i, 0, x} > r_{i, 1, d}, \forall x, i$, and
  \item (A2):  $r_{i, 0, d}+r_{i, 1, d} \leq p-2, \forall i$.
\end{itemize}
Then $\barhatm$ (so in particular, $\barrho$) has an upper triangular crystalline lift with the same Hodge-Tate weights as $\rho$.
\end{thm}
\begin{proof}
By Proposition \ref{shape}, $\barm \in \mathcal E_{\vshape}(\barn_d, \ldots, \barn_1)$. Furthermore by \cite[Cor. 5.10]{GLS14}, $\barhatm \in \mathcal E_{\vtshape}(\barhatn_d, \ldots, \barhatn_1)$.
Let us write $\barn_x =\barm(t_{0, x}, \ldots, t_{f-1, x}; a_x)$.
Then by Proposition \ref{tameinertia}, $t_{i, x} =r_{i, 0, \sigma_{i, 0}(x)}+r_{i, 1, \sigma_{i, 1}(x)}$ where $\sigma_{i, 0}, \sigma_{i, 1}$ are orderings of $\{1, \ldots, d\}$.
By Lemma \ref{rank1}(2), we can find a crystalline lift $\hatN_x$ of $\barhatn_x$ such that $\HT_{i, j}(\hat T(\hatN_x)) =\{ r_{i, j, \sigma_{i, j}(x)} \}$.
Now we claim: $\barhatm$ has an upper triangular lift in $\mathcal E_{\cris}(\hatn_d, \ldots, \hatn_1)$, where $\mathcal E_{\cris}(\hatn_d, \ldots, \hatn_1)$ is the set of crystalline representations made of successive extensions of $\hatn_d, \ldots, \hatn_1$. It is clear that this claim implies our theorem. The proof is similar to \cite[Thm. 7.4]{Gao15}, so we only give a sketch.

We prove the claim by induction on $d$. When $d=1$, there is nothing to prove. Suppose the statement is true for $d-1$, and now consider it for $d$.
Suppose $\barhatm \in \Ext_{\vtshape}(\barhatm_2, \barhatm_1)$
where $\barhatm_2 \in \mathcal E_{\vtshape}(\barhatn_{d}, \ldots, \barhatn_2)$ is of rank $d-1$, and $\barhatm_1$ is of rank 1.
We denote
$$d_{\cris} = \sum_{x=2}^d  \# \{i,  r_{i, 0, \sigma_{i, 0}(x)} > r_{i, 0, \sigma_{i, 0}(1)} \}+
\sum_{x=2}^d \# \{i,  r_{i, 1, \sigma_{i, 1}(x)} > r_{i, 1, \sigma_{i ,1}(1)} \}.$$
Then $\Ext_{\vshape}(\barm_2, \barm_1)$ is a $k_E$-vector space of dimension at most $d_{\cris}$. Since $t_{i, x}$ are all $\leq p-2$, we have $\Ext_{\vtshape}(\barhatm_2, \barhatm_1) \hookrightarrow \Ext_{\vshape}(\barm_2, \barm_1)$ by \cite[Prop. 6.2]{Gao15}. And so $\dim_{k_E}\Ext_{\vtshape}(\barhatm_2, \barhatm_1) \leq d_{\cris}$.

By Lemma \ref{rank1}(2) and our induction hypothesis, we can take crystalline lift $\hatm_1$ (resp. $\hatm_2$) of $\barhatm_1$ (resp. $\barhatm_2$), where $\hatm_2 \in \mathcal E_{\cris}(\hatn_d, \ldots, \hatn_2)$. Then similarly as in \cite[Thm. 7.4]{Gao15}, we will have that $\Ext_{\cris}(\hatm_2, \hatm_1)$ has $\O_E$-free rank equal to $d_{\cris}$, and so we will have an isomorphism of $k_E$-vector spaces:
$$\Ext_{\cris}(\hatm_2, \hatm_1)/\omega_E \to \Ext_{\vtshape}(\barhatm_2, \barhatm_1).$$
And then we can conclude as in \textit{loc. cit.}.
\end{proof}

\subsection{Application to weight part of Serre's conjecture}
For our application to weight part of Serre's conjectures, we will need a special type of crystalline representations.
\begin{defn} \label{definition serre type}
A crystalline representation $V$ as in (CRYS) is called of \emph{Serre type}, if the \emph{auxiliary} Hodge-Tate weights $\HT_{\kappa_{ij}} (V)=\{0, 1, 2, \ldots, d-1\}, \forall j \neq 0$.
\end{defn}
Clearly, when $d=2$, a Serre type representation is precisely a \emph{pseudo-Barsotti-Tate} representation as in \cite[Def. 2.3.1]{GLS15}.

\begin{corollary} \label{Serrecoro}

With notations in (CRYS), suppose $e=2$. Suppose $V$ if of Serre type and $\rhobar$ is upper triangular. Suppose furthermore:
\begin{itemize}
\item $r_{i, 0, x+1} -r_{i, 0, x} \geq d, \forall x, i$.

  \item  $r_{i, 0, d}+(d-1) \leq p-2, \forall i$.
\end{itemize}
Then $\barrho$ has an upper triangular crystalline lift $\rho'$, such that $\HT_{i,j}(\rho')=\HT_{i,j}(\rho), \forall i, j$.
\end{corollary}
\begin{proof}
This is easy consequence of Theorem \ref{mainlocal}. Let us note here from our assumptions, we must have $d^2-1 \leq p-2$.
\end{proof}

Finally we can prove our application in Serre weight conjectures. Since the application is straightforward (by using automorphy lifting theorems in \cite{BLGGT14}), we omit the precise definition of many terms in the statement. The reader can consult \cite{BLGG14} for more detailed explanation of any unfamiliar terms.

\begin{thm} \label{application}
Let $F$ be an imaginary CM field with maximal totally real subfield
  $F^+$, and suppose that $F/F^+$ is unramified at all finite places,
  that every place of $F^+$ dividing $p$ splits completely in $F$,
  and that if $d$ is even then $d[F^+:\Q]/2$ is even. Assume that
  $\zeta_p\notin F$.
  Suppose that $p>2$, and that
  $\overline r:G_F\to\GL_d(\Fpbar)$ is an irreducible
  representation with split ramification. Assume that there is a RACSDC automorphic representation $\Pi$ of
  $\GL_d(\mathbb A_F)$ such that
  \begin{itemize}
  \item $\rbar\cong\rbar_{p,\imath}(\Pi)$ (i.e., $\rbar$ is automorphic).
  \item For each place $w|p$ of $F$, $r_{p,\imath}(\Pi)|_{G_{F_w}}$ is
  potentially diagonalizable.
   \item $\rbar(G_{F(\zeta_p)})$ is adequate.
  \end{itemize}
Assume furthermore that $\rbar \mid_{G_{F_w}}$ is upper triangular for all $w\mid p$.
Let $a = (a_{w})_{w \mid p} \in(\Z^n_+)_0^{\coprod_{w|p}\Hom(k_w,\Fpbar)}$ be a
  Serre weight, and assume that $a\in W^{\rm{cris}}(\rbar)$.

For each $w|p$, suppose the ramification degree $e_w$ of $F_w/\Qp$ is either 1 or 2.
\begin{itemize}
  \item If $e_w=1$, then $a_w$ satisfies the properties in \cite[Thm. 8.3]{Gao15}.
  \item If $e_w=2$, then
  \begin{itemize}
  \item $a_{w, \kappa, x} -a_{w, \kappa, x+1} \geq d-1, \forall \kappa \in \Hom(k_w,\Fpbar), \forall 1\leq x \leq d-1$, and
    \item  $a_{w, \kappa, 1} -a_{w, \kappa, d} \leq p-2-2(d-1), \forall  \kappa \in \Hom(k_w,\Fpbar)$.
  \end{itemize}
\end{itemize}

Then $\rbar$ is automorphic of weight $a$.
\end{thm}

\begin{proof}
Similar as the proof of \cite[Thm. 8.3]{Gao15}, using our Corollary \ref{Serrecoro}.
\end{proof}

\section*{Acknowledgement.}
Our paper is a natural generalization of the results by Toby Gee, Tong Liu and David Savitt, and it is always a great pleasure to acknowledge their beautiful papers.
We also want to sincerely thank Tong Liu for pointing out and helping to correct a serious mistake in Proposition \ref{existbase}.
This paper is written when the author is a postdoc in Beijing International Center for Mathematical Research, and we would like to thank the institute for the hospitality. The author also would like to heartily thank his postdoc mentor, Ruochuan Liu, for his constant interest, encouragement and support. I also thank the anonymous referee(s) for useful comments and corrections.
This work is partially supported by China Postdoctoral Science Foundation General Financial Grant 2014M550539.

\bibliographystyle{alpha}

\begin{thebibliography}{BLGGT14}

\bibitem[BLGG14]{BLGG14}
Thomas Barnet-Lamb, Toby Gee, and David Geraghty.
\newblock Serre weights for ${U}(n)$.
\newblock {\em J. Reine Angew. Math., to appear}, 2014.
\newblock http://arxiv.org/abs/1405.3014.

\bibitem[BLGGT14]{BLGGT14}
Thomas Barnet-Lamb, Toby Gee, David Geraghty, and Richard Taylor.
\newblock Potential automorphy and change of weight.
\newblock {\em Ann. of Math. (2)}, 179(2):501--609, 2014.

\bibitem[Bre99]{Bre99a}
Christophe Breuil.
\newblock Repr\'esentations semi-stables et modules fortement divisibles.
\newblock {\em Invent. Math.}, 136(1):89--122, 1999.

\bibitem[Gao15]{Gao15}
Hui Gao.
\newblock Crystalline liftings and weight part of {S}erre's conjectures.
\newblock {\em preprint, http://arxiv.org/abs/1504.01233}, 2015.

\bibitem[GLS14]{GLS14}
Toby Gee, Tong Liu, and David Savitt.
\newblock The {B}uzzard-{D}iamond-{J}arvis conjecture for unitary groups.
\newblock {\em J. Amer. Math. Soc.}, 27(2):389--435, 2014.

\bibitem[GLS15]{GLS15}
Toby Gee, Tong Liu, and David Savitt.
\newblock The weight part of {S}erre's conjecture for {$\rm{GL}(2)$}.
\newblock {\em Forum Math. Pi}, 3:e2 (52 pages), 2015.

\end{thebibliography}

\end{document}